\documentclass{article}

\usepackage[a4paper, margin=1.3in]{geometry}

\usepackage{amsfonts}
\usepackage{amsmath}
\usepackage{amsthm}
\usepackage{mathrsfs}
\usepackage{color}
\usepackage{bm}
\usepackage{mathtools}
\usepackage{rotating}

\usepackage[tikz]{bclogo}

\usepackage{mdframed}
\mdfdefinestyle{MyFrame}{%
    linecolor=black,
    outerlinewidth=1pt,
    roundcorner=5pt,
    innertopmargin=\baselineskip,
    innerbottommargin=\baselineskip,
    innerrightmargin=10pt,
    innerleftmargin=10pt,
    backgroundcolor=yellow!10!white}

\usepackage[colorlinks,
linkcolor=blue,
anchorcolor = blue,
citecolor=red]{hyperref}

\newtheorem{assumption}{Assumption}
\newtheorem{definition}{Definition}
\newtheorem{theorem}{Theorem}
\newtheorem{lemma}{Lemma}
\newtheorem{proposition}{Proposition}
\newtheorem{remark}{Remark}

\newtheorem{corollary}{Corollary}

\newcommand{\PP}{\big(\mathcal{P}\big)}
\newcommand{\PPx}{\big(\mathcal{P}_x\big)}
\newcommand{\xkm}{x_{k-1}}

\newcommand{\RR}{\mathbb{R}}
\newcommand{\R}{\mathbb{R}}
\newcommand{\NN}{\mathbb{N}}
\newcommand{\CC}{\mathscr{C}}

\newcommand{\argmin}{\mathrm{argmin}\,}
\newcommand{\dom}{\mathrm{dom}\,}

\newcommand{\alg}{{\rm Multiprox}}

\newcommand{\note}[1]{{\color{blue} #1}}

\title{\bf  The multiproximal  linearization  method for convex composite problems\thanks{This work is sponsored by the Air Force Office of Scientific Research under grant FA9550-14-1-0500.}}

\author{J\'er\^ome Bolte\thanks{Toulouse School of Economics (UMR TSE-R, Universit\'e Toulouse Capitole), Manufacture des Tabacs, 31015 Toulouse, France. Email: jbolte@ut-capitole.fr}, Zheng Chen\thanks{School of Aeronautics and Astronautics, Zhejiang University, 310027 Hangzhou, China. Email: z\_chen@zju.edu.cn}, and Edouard Pauwels\thanks{IRIT-UPS, 118 route de Narbonne, 31062 Toulouse, France. Email: edouard.pauwels@irit.fr}}

\begin{document}

\maketitle{}



\begin{abstract}
Composite minimization involves a collection of smooth functions which are aggregated in a nonsmooth manner. In the convex setting, we design an algorithm by  linearizing each smooth component in accordance with its main curvature. The resulting method, called the \alg\ method, consists in solving successively simple problems (e.g.,  constrained quadratic problems) which can also feature some proximal operators. To study the complexity and the convergence of this method, we are led to study quantitative qualification conditions to understand the impact of multipliers on the complexity bounds.   We obtain explicit complexity results of the form $O(\frac{1}{k})$ involving new types of constant terms. A distinctive feature of our approach is  to be able to cope with oracles involving moving constraints. Our method is flexible enough to include the moving balls method, the proximal Gauss-Newton's method, or the forward-backward splitting, for which we recover known complexity results or establish new ones. We show through several numerical experiments how the use of multiple proximal terms can be decisive for problems with complex geometries.

%

\end{abstract}

\section{Introduction}
Proximal methods are at the heart of  optimization. The idea has its roots within the infimal convolution  of Moreau \cite{Moreau:1965} 
with early algorithmic applications to variational inequalities \cite{Martinet:1970}, constrained minimization \cite{rock}, and mechanics \cite{Moreau}.  The principle is elementary but far reaching: it simply consists in generating algorithms by considering  successive strongly convex 
approximations of a given objective function. Many methods can be seen through these lenses, like for instance, the classical gradient method, the 
gradient projection method, or mirror descent methods \cite{rosen1960gradient,rosen1961gradient,Levitin:66,Nemirovskii:1983,auslender2006interior}. At this day, the  most famous example is probably  the forward-backward 
splitting algorithm \cite{passty1979ergodic,lions1979splitting,Tseng:1991,Combettes:2005} and its accelerated variant FISTA 
\cite{Beck:2009}. Many generalizations in many settings followed, see for instance \cite{eckstein1993nonlinear,Ortega:2000,Combettes:2011,leroux2012stochastic,salzo2012convergence,villa2013accelerated,plc13,plc16,schmidt2017minimizing,BC17}.

{
In order to deal with problems with more complex structure, we are led to consider models of the form 
\begin{eqnarray}
\PP:\ \ \ {\mathrm{min}}\left\{ g\left(F(x)\right):{x\in\mathbb{R}^n}\right\},\nonumber
\end{eqnarray} 
where $F=(f_1,\ldots,f_m)$ is a collection of convex differentiable functions with $L_i$ Lipschitz continuous gradient and  $g:\mathbb{R}^m\rightarrow \mathbb{R}\cup\{+\infty\}$ is a proper convex lower semicontinuous function. 
{The function $g$ is allowed to take infinite values offering a great flexibility in the modeling of constraints while it is often assumed to be finite in the literature.}
{When $g$ is restricted to be Lipschitz, a} natural proximal approach to this problem consists in  linearizing the smooth part, leaving the nonsmooth term unchanged and adding an adequate quadratic  form. Given $x$ in $\RR^n$, one obtains the proximal Gauss Newton method or the prox-linear method
$$\mbox{\rm(PGNM)}\quad x^+=\underset{y\in \RR^n}\argmin \left(g\left(F(x)+\nabla F(x)(y-x)\right)+\frac{\lambda}{2}\|y-x\|^2\right),$$ where 
\begin{equation}\label{ineq1}\lambda\geq {L_g}\displaystyle\max_{1,\ldots,m} L_i
\end{equation}
{with $L_g$ being the Lipschitz constant of $g$.}
The method\footnote{Also known as the proximal Gauss-Newton's method} progressively emerged from quadratic programming, see \cite{pshe87} and references therein, but also from ideas \`a la Gauss-Newton \cite{flet80,Burke:1985,Burke:1995}. It was eventually formulated under a proximal form in \cite{lewis2015proximal}. It  
allows to deal with general nonlinear programming pro\-blems and unifies within a simple framework many different classes of 
methods met in practice \cite{Li:2002,Li:2007,Cartis:2001,lewis2015proximal,BP,Pauwels:2016,Drusvyatskiy:2016}. It is 
{one of the rare primal methods}
for composite problems without linesearch\footnote{Indeed, PGNM is somehow a  ``constant step size" method} (see also \cite{Auslender:10}), and as such assessing its  complexity is a natural question.
   Even though the complexity analysis of convex first order methods has now become  classical (see e.g.,  \cite{Nesterov:2004}), considerable difficulties remain for com\-po\-si\-te problems with such 
   generality. One of the 
   {reasons}
    is that constraints, embodied in $g$, generate multipliers whose role is not yet understood. To our knowledge, there are very few works in this line.  In \cite{Drusvyatskiy:2016} the 
authors study this method under error bounds conditions and establish linear convergence results, see also \cite[Section 2.3]{Nesterov:2004} for related results. In a recent   
article \cite{Drusvyatskiy:20162}, the authors propose an acceleration of the same method and they obtain faster convergence guaranties under mild assumptions. The global complexity  with general assumptions on a convex $g$ seems to be an open question. 

\smallskip

We work here along a different line and we propose a new flexible method with {\em inner quadratic/proximal approximations}. Given a starting point $x$, we consider
$$(\alg)\quad x^+=\underset{y\in\RR^n}\argmin \,g\left(F(x)+\nabla F(x)\left(y-x\right)+\|y-x\|^2\left(\begin{array}{c}\frac{\lambda_1}{2}\\ \vdots \\ \frac{\lambda_m}{2}\end{array}\right)\right),$$
or more explicitly
$$ x^+=\underset{y\in\RR^n}{\argmin} \,g\left(\begin{array}{c}f_1(x)+
{\nabla f_1(x)^T (y-x)} 
+\lambda_1\frac{\|y-x\|^2}{2}\\ \vdots \\ f_m(x)+
{\nabla f_m(x)^T (y-x)} 
 +\lambda_m\frac{\|y-x\|^2}{2}\end{array}\right),$$
where \begin{equation}\label{ineq2}\lambda_i\geq L_i\mbox{ for }i=1,\ldots,m,\end{equation}
{and $g$ is allowed to be extended valued.}
  In order to preserve the convexity properties of the local model, the function $g$ is assumed to be componentwise nondecreasing. 
{In spite of the monotonicity restriction on $g$,}
  this model is quite versatile and includes  as special cases general inequality constrained convex programs, important min-max problems or additive composite models.  Observe that the local approximation used in \alg\, is sharper than the one in the proximal Gauss-Newton's method since it relies on the vector $(L_1,\ldots,L_m)$ rather than on the mere constant $\max\{L_i:i=1,\ldots,m\}$. {Due to the presence of multiple proximal/gradient terms we called our method \alg.} The key idea behind \alg, already present in \cite{Auslender:10}, is to design  local approximations through upper quadratic models specifically taylored for each of the components. This makes the method well adapted to the geometry of the original problem and allows in general to take much larger and clever steps as illustrated in numerical experiments in the last section. 

}

Studying this method presents several serious difficulties. First $g$ 
{may not have}
  full domain (contrary to what is assumed in \cite{Burke:1995,Li:2002,Drusvyatskiy:2016,Drusvyatskiy:20162}), which reflects the fact that subproblems may feature ``moving constraint sets". Even though moving constraints are very common in sequential convex programming, we did not find any genuine results on complexity in the literature.
Secondly the nature of our algorithm rises new issues concerning qualification conditions (and subsequently on the role of Lagrange multipliers in the complexity analysis). The qualification condition we consider is surprisingly simple to state, yet non trivial to study:
$$F^{-1}(\mbox{int}\,\dom g)\neq \emptyset.$$
This Slater's like condition is specific to situations when $g$ is monotone and was already used in \cite[Theo\-rem~2]{hiriart2006note} to provide formulas for computing the Legendre transform of com\-po\-si\-te functions and in \cite{BGW} to study stable duality and chain rule. Under this condition, we establish a complexity result of the form $O(\frac{1}{k})$ whose constant term depends on the geometry of $\PP$ through a quantity combining various curvature constants of the components of $F$ with the multipliers attached to the subproblems. When $g$ is finite and Lipschitz continuous, the complexity boils down to
$$\frac{L_g\sqrt{\sum_{i=1}^mL_i^2}}{2k}\|x_0-x_*\|^2$$
where $L_g$ is the Lipschitz constant of $g$. The exact same analysis leads to improved bounds if the outer function $g$ has a favorable structure, such as the 
{coordinatewise} maximum, in which case, the numerator reduces to $\max_{i=1,\ldots,m}\{L_i\}$. To our knowledge, these results were missing from the {literature}. 

We study further the boundedness properties of the sequences generated by {\alg .}
We derive in particular quantitative bounds on the multipliers for hard constrained\footnote{Here, hard constraints means that only feasible point can be considered, contrasting with infeasible methods (e.g., \cite{auslender2013extended,toint14}).} problems. This allows us  in turn to derive complexity estimates for sequential convex methods in  convex nonlinear programming such as the moving balls method \cite{Auslender:10} and its nonsmooth objective variant \cite{shefi2016dual}. To put this into perspective,  the only feasible methods for nonlinear programming  which come with such explicit estimates  are, to the best of our knowledge, interior point methods  (see \cite{nesterov1994interior,ye} and references therein).

We also analyze into depth the important cases when Lipschitz constants are not known\footnote{We refer to constants relative to the gradients.} or when they only exist locally (e.g.,  the $C^2$ case). In this setting, the ``step sizes" (the various $\lambda_i$) cannot be tuned a priori and thus complexity results are much more difficult to establish due to  the use of linesearch routines. Yet we obtain some useful rates and we are able to establish convergence of the sequence. Once more we insist on the fact that convergence in this setting is not an easy matter and very few results are known \cite{Auslender:10,auslender2013extended,solodov2009global,Cartis:2001,toint14}.

Finally, we illustrate the efficiency of \alg\  on synthetic data. We consider a composite function consisting of the maximum of convex quadratic functions with different smoothness moduli. We compare our method with the proximal Gauss-Newton algorithm and its accelerated variant described in \cite{Drusvyatskiy:20162}. These experiments illustrate that, although the complexity estimates of the \alg\ algorithm are not better than existing estimates for the 
{concurrent} methods, its adaptivity to different smoothness moduli gives it a crucial advantage in practice.

\paragraph{Outline.}  In Section \ref{SE:problem} we describe the composite optimization problem and study  qualification conditions.  Section \ref{SE:complexity} provides first general complexity and convergence results and presents consequences for specific models. Section \ref{s:search} on linesearch describes cases when Lipschitz constants are unknown or merely locally bounded. In the section \ref{SE:numerical} we provide numerical experiments illustrating the efficiency of our method.

\paragraph{Notations}
$\mathbb{R}^n$ is the $n$-dimensional Euclidean space equipped with 
the Euclidean norm $\|\cdot\|$. For $x \in \RR^n$ and $r \geq 0$, $B(x,r)$ denotes the closed Euclidean ball of radius $r$ centered at $x$. By $\mathbb{R}^n_+$, we  denote the $n$-dimensional nonnegative orthant ($n$-dimensional vectors with nonnegative entries). The notations of $<$, $\leq$, $>$, and $\geq$ between vectors indicate that the corresponding inequalities are met coordinatewise.

Our notations for convex analysis are taken from  \cite{Rockafellar:1998}. We recall the most important ones. Given a convex extended-valued convex function $h:\mathbb{R}^m \rightarrow \mathbb{R}\cup \{+\infty\}$, we set
$$\dom h := \{z\in\mathbb{R}^m: h(z) < +\infty\}.$$
The subdifferential of $h$ at any $\bar{z} \in \dom h$ is defined as usual by
$$\partial h(\bar{z}) := \{\lambda \in\mathbb{R}^m: 
{(z - \bar{z})^T \lambda}
 + h(\bar{z}) \leq h(z), \forall z\in\mathbb{R}^m\},$$
and is the empty set if $z \not \in \dom h$. 

{For any convex subset $D\subset \R^m$, we denote  by $i_D$ its indicator function -- recall that $i_D(x)=0$ if $x$ is in $D$, $i_D(x)=+\infty$ otherwise.}

\section{Minimization problem and algorithm}\label{SE:problem}

\subsection{Composite model and assumptions} 
We consider a composite minimization problem of the type:
\begin{eqnarray}
\PP:\ \ \ {\mathrm{min}}\left\{ g(F(x)):{x\in\mathbb{R}^n}\right\},\nonumber
\end{eqnarray} 
where $g:\mathbb{R}^m\rightarrow \mathbb{R}\cup\{+\infty\}$ and $F:\mathbb{R}^n\rightarrow \mathbb{R}^m$. We set $F=(f_1,\ldots, f_m)$ and we make the standing assumptions:\\
\begin{assumption}
\begin{enumerate}
\item[(a)] Each $f_i$ is continuously differentiable, convex, with $L_i$ Lipschitz continuous gradient {\rm($L_i\geq 0$)}.
\item[(b)]  The function $g:\mathbb{R}^m\rightarrow \mathbb{R}\cup\{+\infty\}$ is convex, proper, lower semicontinuous and $L_g$ Lipschitz continuous on its domain. That is 
$$|g(x)-g(y)|\leq L_g\|x-y\| \text{ for all $x,y$ in $\dom g$.}$$

For each $i=1,\ldots,m$ with $L_i > 0$, then $g$ is nondecreasing in its $i$-th argument\footnote{{For any such $i$ and any $m$ real numbers $z_1, \ldots z_m$, the function $z \mapsto g(z_1,\ldots, z_{i-1}, z,z_{i+1}, \ldots, z_m)$ is nondecreasing. In particular, its domain is either the whole of $\RR$ or a closed half line $(-\infty, a]$ for some $a\in \RR$, or empty. }}. In other words $g$ is nondecreasing in its $i$-th argument  whenever $f_i$ is not affine.
\end{enumerate}
\label{AS:convex}
\end{assumption}

{\begin{remark} {\rm (a) Note that the monotonicity restriction on $g$ implies some restrictions. For example, ignoring the affine components of $F$, for any $z \in \dom g$, we also have $z - \RR_+^m \subset \dom g$, so that $\dom g$ is not compact. Prominent examples for $g$ includes the max function, the indicator of  $\RR_-^m$ (which allows to handle nonlinearity of the type $f_i\leq 0$), support functions of a subset of positive 
{numbers}. Depending on the structure of $F$ other examples are possible (see further sections). 
\\
(b) Contrary to  $L_1,\ldots, L_m$, the value of $L_g$ is never required to design/run the algorithm (see in particular Theorems \ref{TH:nonincreasing} and \ref{TH:nonincreasingBack}).}\end{remark}}
\begin{assumption}
				The function $g\circ F$ {is proper and }has a minimizer.
\label{AS:effectiveness}
\end{assumption}
Observe that the monotonicity of $g$ and the convexity  of $F$ in Assump\-tion~\ref{AS:convex} ensure that the problem $\PP$ is a convex optimization problem, in other words:
\begin{equation}
\text{ $g\circ F$ is convex.}
\label{LE:convex}
\end{equation}

\medskip

\subsection{The Multiproximal linearization algorithm}

\noindent  Let us introduce the last fundamental ingredient necessary to the description of our method:
$$\bm{L} := (L_1,\ldots,L_m)^T\in \RR_+^m.$$ 
Observe that the monotonicity of $g$ in Assumption \ref{AS:convex} implies that for any $z\in \dom g$, one has 
\begin{equation}\bm{L}^T\lambda \geq 0, \text{ for any }\lambda \in \partial g(z).\label{LE:Lnv_positive}
\end{equation}
\medskip

The central idea  is to use quadratic upper approximations  componentwise on the smooth term~$F$. We thus introduce the following mapping  
\begin{align*}
	H(x,y) := F(x) + \nabla F(x) (y - x) + \frac{\bm{L}}{2}\|y - x\|^2,\nonumber \quad (x,y)\in\mathbb{R}^n\times\mathbb{R}^n, \end{align*}
where $\nabla F$ denotes the Jacobian matrix of $F$. This leads to the following family of subproblems 
\begin{eqnarray}
\PPx: \quad  \min\{ \,g (H(x,y)):y\in \RR^n\}.\nonumber
\end{eqnarray}
where $x$ ranges in $F^{-1}(\dom g)$.
As shall be discussed in further sections this problem is well-posed for broad classes of examples. We make the following additional standing assumption:
 \begin{assumption}\label{AS:existence_sub}
For any $x$ in $F^{-1}(\dom g)$, the function $g\circ H(x,\cdot)$ has a minimizer.
\end{assumption}

\noindent
Elementary but important properties of problem $\PPx$ are given in the following lemma.
\begin{lemma}
\label{LE:sub_property}
For any $x\in F^{-1}(\dom g)$, the following statements hold:
\begin{description}
\item (1) $\dom g(H(x,\cdot)) \subset \dom g\circ F$ and
 $$g(F(y)) \leq g(H(x,y)), \: \forall y \in \dom g(H(x,\cdot)).$$
\item (2) $g(F(x)) = g(H(x,x))$.
\item (3) $g\circ H(x,\cdot)$ is proper and convex.
\end{description}
\end{lemma}

\begin{proof}

According to Assumption \ref{AS:convex} and the descent Lemma, \cite[Lemma 1.2.3]{Nesterov:2004},  for every $i$ in $\{1,\ldots,m\}$ one has
$$f_i(y) 
\begin{cases}
\leq f_i(x) +
 { \nabla f_i(x)^T(y - x)} + \frac{L_i}{2}\|y - x\|^2,\ L_i>0,\\
= f_i(x) + 
 { \nabla f_i(x)^T(y - x)} + \frac{L_i}{2}\|y - x\|^2,\ L_i=0,
\end{cases}\ \ \forall (x,y)\in\mathbb{R}^n\times\mathbb{R}^n.
$$
By the monotonicity properties of $g$ we obtain that  $g(F(y)) \leq g(H(x,y))$ for all $y\in \mathbb{R}^n$ and (1) is proved. 
Items (2) and (3) follow from simple verifications. \end{proof}


\smallskip

 \begin{mdframed}[style=MyFrame]
 \begin{center}
 {\bf Multiproximal method (\alg)}
 \end{center}
 \bigskip
\noindent \ \ \ \ \ Choose $\ x_0 \in F^{-1}(\dom g)$ and iterate for $k\in \NN$:
\begin{equation}
x_{k+1}\in \underset{y\in \RR^n}{\argmin}\ \ g\left(F(x_k) + \nabla F(x_k) (y - x_k) + \frac{\bm{L}}{2}\|y - x_k\|^2\right)
\label{Numerical_Scheme}
\end{equation}
\mbox{with the choice $x_{k+1}=x_k$ whenever $x_k$ is a minimizer of {$g(F(x_k,\cdot))$}.}
\end{mdframed}

%
\noindent
If we set 
$$p(x):= \underset{ y\in\mathbb{R}^n}{\text{argmin}}\ g (H(x,y))$$
 for any $x$ in $F^{-1}(\dom g)$ the algorithm simply {reads as}
 $x_{k+1}\in p(x_k)$ with $x_{k+1}=x_k$ whenever $x_k\in p(x_k)$.

\begin{remark}{\rm (a) Item (1) of Lemma~\ref{LE:sub_property}, {along with Assumption \ref{AS:existence_sub},}  implies that the algorithm is well defined. Observe that Lemma~\ref{LE:sub_property} actually shows that our  algorithm is based on the classical idea of minimizing successively majorant functions coinciding at order 1 with the original function.\\
(b) As already mentioned, the algorithm does not require the knowledge of the Lipschitz constant of $g$ on its domain.}
\end{remark}

{
\subsection{Examples and implementation issues}\label{examples}

We give here two important examples for which the subproblems are simple quadratic problems.

\subsubsection{Convex nonlinear programming}\label{movi} Consider the classical 
convex 
nonlinear programming problem 
\begin{equation}\label{exnlp}\min \left\{f(x):f_i(x)\leq 0,\:  i=1,\ldots,m, \, x \in \R^n \right\},\end{equation}
where $f \colon \RR^n \to \RR$ is a convex function with $L_f$ Lipschitz continuous gradient and each $f_i$ is defined as in Assumption \ref{AS:convex}.  Using the reformulation:
 $ \min \{ f(x)+i_{\RR_-^m}(f_1(x),\ldots, f_m(x)): x\in \RR^n\}$ and setting $g(s,y_1,\ldots,y_m)=s+i_{\R_-}(y_1,\ldots,y_m)$ and $F(x)=(f(x),f_1(x),\ldots,f_m(x))$, the problem\footnote{There is a slight shift in the indices of $F$} \eqref{exnlp} can be seen as an instance of $\PP$. \alg\ writes 
\begin{align}
	x_{k+1} \in \underset{y\in\mathbb{R}^n}{\argmin}\quad & f(x_k) + 
	\nabla f(x_k)^T (y- x_k)
	+ \frac{L_f}{2}\|y - x_k\|^2  \nonumber \\
	\mathrm{s.t.}\quad&f_i(x_k) +
	\nabla f_i(x_k)^T (y- x_k)
	  + \frac{L_i}{2}\|y - x_k\|^2 \leq 0,\, i=1,\ldots,m,
	\label{EQ:MBalgo}
\end{align}
which is a  generalization of the moving balls method \cite{Auslender:10,shefi2016dual} in the sense that our algorithm offers the additional flexibility that affine constraints can be left unchanged in the subproblem (by setting the corresponding $L_i$ to $0$). Assume for simplicity that $L_f>0$.

Computing $x_{k+1}$ leads to solve very specific quadratic problems. Indeed, if $q$ is  a quadratic form appearing within the above subproblem, its Hessian is given by  $\nabla^2q=cI_n$ (with $c>0$) or $\nabla^2q=0_n$, where $I_n$ (resp. $0_n$) denotes the identity (resp. null) matrix in $\R^{n\times n}$. Computing $x_{k+1}$ amounts to computing
 the Euclidean projection of a point to an intersection of Euclidean balls/hyperplanes. Both types of sets have extremely simple projection operators and one can thus apply Dykstra's projection algorithm (see e.g., \cite{BC17}) or a fast quadratic solver (see e.g., \cite{nocedal}).  Let us also mention that this type  of problems can be treated very efficiently by specific methods based on activity detection  described in \cite{Auslender:10, shefi2016dual}.

\subsubsection{
Min-max problems}\label{minmax} 
We consider the problem 
$$\min\left\{\max_{i=1,\ldots,m} f_i (x):x\in \RR^n\right\}.$$ This type of problems is very classical in optimization but also in game theory (see e.g.,  \cite{Nesterov:2004}). Observing that $g=\max_{1,\ldots,m} {f_i}$ satisfies our assumptions, we see that the problem is already under the form $\PP$. The substeps assume thus the form 
 \begin{align}
	x_{k+1} \in \underset{y\in\RR^n}{\argmin}\quad & \left(\max_{i=1,\ldots,m}(f_i(x_k) + 
	\nabla f_i(x_k)^T(y - x_k)
	 + \frac{L_i}{2}\|y - x_k\|^2)\right) .\nonumber
\end{align}
As previously explained, this subproblem can be rewritten as a simple quadratic problem and it can thus be solved through the same means. In the last section, we illustrate the numerical efficiency of \alg\ on this type of problems.


\begin{remark} \rm Other cases can be treated by \alg. Consider for example, the following problem
\begin{align*}
	\underset{x\in\mathbb{R}^n}{\mathrm{min}}\quad\quad& \mathrm{max}\{f_1(x),f_2(x)\} + \|x\|_1,\\
	\mathrm{s.t.}\quad\quad&\mathrm{max}\{f_3(x),f_4(x)\} + \|x\|_{\infty} \leq 0.
\end{align*}
where $f_i$ are smooth convex functions for $i = 1,\ldots,4$. Then,   for any $x \in \RR^n$, the solution of $\PPx$ can be computed as follows:
\begin{align*}
				\min_{y \in \RR^n, \,s \in \RR, \,t\in \RR^n, \,u \in \RR, \,v \in \RR} \quad\quad &s + \sum_{i=1}^n t_i\\
				\mathrm{s.t.}\quad\quad &f_1(x) + \nabla f_1(x)^T (y - x) + \frac{L_1}{2} \|y - x\|^2 \leq s\\
				&f_2(x) + \nabla f_2(x)^T (y - x) + \frac{L_2}{2} \|y - x\|^2 \leq s\\
				&x_i\leq t_i, \, i=1,\ldots, n\\
				&-x_i\leq t_i, \, i=1,\ldots, n\\
				&{t_i\leq v}, \, i=1,\ldots, n\\
				&f_3(x) + \nabla f_3(x)^T (y - x) + \frac{L_3}{2} \|y - x\|^2 \leq u\\
				&f_4(x) + \nabla f_4(x)^T (y - x) + \frac{L_4}{2} \|y - x\|^2 \leq u\\
				&u + v \leq 0
\end{align*}
which is a  quadratically constrained linear program. 

\end{remark}

}

\subsection{Qualification, optimality conditions and a condition number}\label{s:cr}
The first issue met in the study of \alg\ is the one of qualification conditions both for $\PP$ and $\PPx$. Classical qualification conditions take the form
\begin{align}
				\label{EQ:qualifTradi}
				N_{\dom g}(F(x)) \cap \mathrm{ker}(\nabla F (x)^T) = \{0\}, \quad \forall x \in F^{-1}(\dom g),
\end{align}
where $N_{\dom g}(F(x))$ denotes the normal cone to $\dom g$ (see e.g.,  \cite[Example 10.8]{Rockafellar:1998}). In this section, we first describe a different qualification condition which takes advantage of the specific monotonicity properties of $g$ as described in Assumption \ref{AS:convex}. This condition was already used in \cite{hiriart2006note} to provide a formula for the Legendre conjugate of the composite function $g(F(\cdot))$. In this setting, we show that this condition allows to use the chain rule  and provides optimality conditions which will be crucial to study \alg\ algorithm. We emphasize that this qualification condition is much more practical than conditions of the form (\ref{EQ:qualifTradi}). Besides it is also naturally amenable to quantitative estimation which appears to be fundamental for the computation of complexity estimates.

\paragraph{Qualification condition and chain rule} 
The following Slater's like qualification condition is specific to the ``monotone composite model" we consider here (see \cite[Theorem 2]{hiriart2006note} and \cite[ Section 3.5.2]{BGW}).

{We make the following standing assumption:}\begin{assumption}[Qualification]\label{AS:Qualification}
				There exists $\bar{x}$ in $\displaystyle F^{-1}(\text{\rm int}\,\dom g).$
\end{assumption}
\noindent
The following result illustrates the main interest of Assumption~\ref{AS:Qualification}:
\begin{proposition}[Chain rule]\label{PR:chainRule}
For all $x$ in $F^{-1}(\dom g)$:
$$\partial \left(g\circ F\right)(x)=
{\nabla F(x)^T}
\partial g(F(x)).$$
\end{proposition}
Proposition \ref{PR:chainRule} follows from \cite[Theorem 3.5.2]{BGW}, and we provide a self contained proof in  Appendix \ref{SE:appendixProofChainRule}. Another interesting and useful consequence of Assumption \ref{AS:Qualification} is that it automatically ensures a similar qualification condition for all subproblems~$\PPx$.
\begin{proposition}[Qualification for subproblems]\label{PR:qualifSub}
For all $x$ in $F^{-1}(\dom g)$ there exists $w(x) \in \RR^n$ such that
\begin{align*}
				H(x, w(x)) \in \mathrm{int}\ \dom g.			
\end{align*}
\end{proposition}
\begin{proof}
	An explicit construction of $w(x)$ is provided in Lemma \ref{LE:lower_bound_Hxw} (Appendix \ref{SE:appendixExplicitBound}).
\end{proof}
These results provide necessary and sufficient optimality conditions for $\PP$ and for $\PPx$.

\begin{corollary}[Fermat's rule]
\label{TH:optimality_condition}
A point $x_* \in F^{-1}(\dom g)$ is a minimizer of $g\circ F$ if and only if 
\begin{eqnarray}
\exists \lambda_*\in \partial g(F(x_*))\ \text{s.t.}\ 
{\nabla F(x_*)^T}
 \lambda_* = 0.\nonumber
\end{eqnarray}
For any $x$ in $F^{-1}(\dom g)$, and for any $y$ in $\mathbb{R}^n$, we have $y \in p(x)$ if and only if
 $$\exists \nu\in\partial g(H(x,y))\ \text{s.t.}\ 
{\nabla_y H(x,y)^T}
  \nu = 0.$$
\end{corollary}
\noindent
The first part of the corollary is of course an immediate consequence of the chain rule in Proposition~\ref{PR:chainRule}. The second part holds true for the same reasons, replacing Assumption \ref{AS:Qualification} by Proposition~\ref{PR:qualifSub}.

\paragraph{Lagrange multipliers and a condition number}

\begin{definition}[Lagrange multipliers for $\PPx$]
\label{DE:set}
For any $x\in F^{-1}(\dom g)$ and any $y\in p(x)$, we set
\begin{eqnarray*}
\mathcal{V}(x,y) := \big\{\nu\in\partial g(H(x,y)): 
{\nabla_y H(x,y)^T}
\nu = 0\big\}.
\end{eqnarray*}
\end{definition}

\smallskip

 The following quantity, which can be seen as  a kind of condition number captures  the boundedness properties of the multipliers 
 for the subproblems.  It will play a crucial role in our complexity studies.
\begin{lemma}[A condition number]
\label{LE:bounded1}
Given any nonempty compact set $K\subseteq F^{-1}(\dom g)$ and any $\gamma \geq \min g\circ F$, the following quantity is finite 
\begin{eqnarray*}
	\CC_\gamma(K):= \sup\left\{\inf \{\bm{L}^T \nu : \nu\in \mathcal{V}(x,y)\}:\:x\in K,\, g(F(x)) \leq \gamma,\, y\in p(x) \right\} .
\end{eqnarray*}
\end{lemma}
\begin{proof}
		We shall see that		Lemma \ref{TH:explicitBound} provides an explicit bound on this condition number and as a consequence
				 $\CC_\gamma(K)$ is finite.
\end{proof}

\begin{remark}[$\CC_\gamma(K)$ as a condition number] {\rm {In numerical analysis, the term condition number usually refers to a measure of the magnitude of variation of the output as a function of the variation of the input.  For instance, when one studies the usual gradient method for minimizing a convex function $f:\R^n\to\R$ with Lipschitz gradient $L_f$, one is led to algorithms of the form 
$x_{k+1}=x_k-\nabla f(x_k)/L_f$
 and the complexity takes the form $\displaystyle L_f\|x_0-x^*\|^2/(2k)$ where $x^*$ is a minimizer of $f$. 
The bigger $L_f$ is, the worse the estimate is.
  It captures in particular the compositional structure of the model by combining the smoothness modulus of $F$ with some regularity for $g$ captured through KKT multipliers.}}\end{remark}

\section{Complexity and convergence}\label{SE:complexity}

This section is devoted to the exposition of the complexity results obtained for \alg. In the first {subsection}, we describe our main results, an abstract convergence result and  
{we provide explicit} complexity estimate. We then describe the consequences for known algorithms.

\subsection{Complexity results for \alg}

\paragraph{General complexity and convergence results}

We begin by establishing that \alg\ is a descent method:

\begin{lemma}[Descent property]
\label{LE:nonincreasing}
For any $x\in F^{-1}(\dom g)$, if $y\in p(x)$, one has
\begin{equation}
g(F(y)) - g(F(x)) \leq -\frac{\|y - x\|^2}{2} \bm{L}^T {\nu},\nonumber
\end{equation}
for all $\nu$ in $\mathcal{V}(x,y)$.
\end{lemma}
\begin{proof}
By Definition \ref{DE:set}, for every $x\in F^{-1}(\dom g)$ and $y\in p(x)$ we have $
{\nabla_y H(x,y)^T}
{\nu} = 0$, for any $\nu\in\mathcal{V}(x,y)$. In other words
\begin{eqnarray}
{\nabla F(x)^T}
 {\nu} +   (y - x)(\bm{L}^T {\nu}) = 0,\ \forall \nu\in\mathcal{V}(x,y).
\label{EQ:A_suboptimality}
\end{eqnarray}
By convexity of $g$, one has  
\begin{eqnarray}
 g(H(x,y)) - g(F(x))&\leq& [H(x,y) - F(x)]^T {\nu}\nonumber\\
&=& 
{[\nabla F(x)^T {\nu}]^T(y - x)}
+ \frac{\|y - x\|^2}{2} \bm{L}^T {\nu},\ \forall \nu \in\mathcal{V}(x,y).\nonumber
\end{eqnarray}
Substituting this inequality into equation (\ref{EQ:A_suboptimality}) yields
\begin{eqnarray}
 g(H(x,y)) - g(F(x)) \leq - \frac{\|y - x\|^2}{2} \bm{L}^T {\nu},\ \forall \nu\in\mathcal{V}(x,y).\nonumber
 \end{eqnarray}
Combining Lemma \ref{LE:sub_property} with this inequality completes the proof.
\end{proof}
\begin{remark}[\alg\ is a descent method]{\rm 
For any sequence $(x_k)_{k\in\mathbb{N}}$ generated by \alg, the corresponding sequence of objective values $(g\circ F(x_k))_{k\in\mathbb{N}}$ is  nonincreasing.}
\label{RE:nonincreasing}
\end{remark}

\noindent
Let us set 
\begin{eqnarray*}
S & := &\argmin g\circ F\\ 
  &= & \{x\in F^{-1}(\dom g): \exists \lambda \in \partial g(F(x))\ \text{s.t.}\ 
 {\nabla F(x)^T}
   \lambda = 0\}.\nonumber
\end{eqnarray*}

The following theorem is our first main result under the assumption that the smoothness moduli of the components of $F$ are known and available to the user. The first item is a complexity result while the second one is a convergence result. Discussion regarding the impact of our results {on}
 other existing algorithms is held in Section \ref{SE:conseqComplexity}. 
\begin{theorem}[Complexity and convergence for \alg]
Let $(x_k)_{k\in\mathbb{N}}$ be a sequence generated by \alg. Then, the following statements hold:
\begin{description}
	\item (i) For any $x_*\in S$ set $B_{x_0,x_*}= B(x_*,\|x_0-x_*\|)$.
Then, for all $k \geq 1$,  
\begin{equation}
g (F(x_k)) - g(F(x_*)) \leq\frac{\CC_{g(F(x_0))}\left(B_{x_0,x_*}\right)}{2k} \|x_0 - x_*\|^2;
\label{EQ:complexity_f}
\end{equation}
\item (ii) The sequence $(x_k)_{k\in\mathbb{N}}$ converges to a point in the solution set $S$.\end{description}
\label{TH:nonincreasing}
\end{theorem}
\begin{proof}
(i) {Let $n$ be a  positive integer.}  
The following elementary observation appears to be very useful:  given any function of the form
$$f:\mathbb{R}^n\rightarrow \mathbb{R},\ x\mapsto a \|x\|^2 +  {b^T x} + c,$$
where $a\in \mathbb{R}_+$, $b\in\mathbb{R}^n$, and $c\in\mathbb{R}$, if there exists $\hat{x}$ in $\mathbb{R}^n$ such that $\nabla f(\hat{x}) = 0$, one has
\begin{eqnarray}
f(x) = f(\hat{x}) + a\|x - \hat{x}\|^2\ \forall x\in\mathbb{R}^n.
\label{EQ:strong_a}
\end{eqnarray}
By Definition \ref{DE:set}, {for any integer $k>0$ and} for every  $\nu_k\in\mathcal{V}(x_{k-1},x_k)$ the gradient of $H(x_{k-1},\cdot)^T {\nu}_{k}$ at $x_k$ is zero, i.e., $\nabla_y [H(x_{k-1},x_k)^T {\nu}_{k}] =0$. Combining the explicit expression of $H(x_{k-1},\cdot)^T {\nu}_{k}$ with equation (\ref{EQ:strong_a}) and considering that $\bm{L}^T\nu_k \geq 0$ (see \eqref{LE:Lnv_positive}), one has, for any $\nu_k$ in $\mathcal{V}(x_{k-1},x_k)$ and any $x_* $ in $S$,
  \begin{eqnarray}
	H(x_{k-1},x_{k})^T {\nu}_{k} & =  & F(x_{k-1}){^T\nu_k}+[\nabla F(x_{k-1})(x_k-x_{k-1}){]^T\nu_k}+\frac{\|x_k-\xkm\|^2}{2}\bm{L}{^T}\nu_k\\
 & = &H(x_{k-1},x_{*})^T {\nu}_{k} - \frac{\bm{L}^T {\nu}_{k}}{2}{\|x_k - x_*\|^2}. \label{EQ:strong_equation} \end{eqnarray}
 
 \noindent
 {According to} the convexity of $g$, for any $\nu_k$ in $\mathcal{V}(x_{k-1},x_k)$ and any $x_*$ in $S$,
\begin{eqnarray}
g\circ H(x_{k-1},x_{k}) - g\circ F(x_{*})\leq  [H(x_{k-1},x_{k})  - F(x_{*}) ]^T {\nu}_{k}.
\label{EQ:complexity_1}
\end{eqnarray}
As a consequence, for any $\nu_k$ in $\mathcal{V}(x_{k-1},x_k)$ and any $x_*$ in $S$, one has
\begin{align}
\label{EQ:bound_decreasing}
\begin{split}
&\ g\circ F(x_{k})  - g\circ F(x_{*})\\
\overset{(a)}{\leq} &\ [H(x_{k-1},x_k) - F(x_*)]^T {\nu}_k\\
\overset{(b)}{=} &\ H(x_{k-1},x_{*})^T {\nu}_{k}   - \frac{\bm{L}^T {\nu}_{k}}{2}{\| x_k - x_*\|^2} - F(x_{*})^T {\nu}_{k}  \\
\overset{(c)}{=} &\ [ F(x_{k-1}) + \nabla F(x_{k-1})( x_{*} - x_{k-1}) ]^T {\nu}_{k}     + \frac{\bm{L}^T   {\nu}_{k} }{2} (\|x_{k-1} - x_*\|^2 - \|x_k - x_*\|^2) - F(x_{*})^T  {\nu}_{k} \\
\overset{(d)}{\leq} &\ F(x_{*})^T {\nu}_{k}   + \frac{\bm{L}^T   {\nu}_{k} }{2} (\|x_{k-1} -x_*\|^2 - \|x_k - x_*\|^2) - F(x_{*})^T  {\nu}_{k}  \\
= &\ \frac{\bm{L}^T   {\nu}_{k} }{2} (\|x_{k-1} -x_*\|^2 - \|x_k - x_*\|^2),
\end{split}
\end{align}
where $(a)$ is obtained by combining Lemma \ref{LE:sub_property} with equation (\ref{EQ:complexity_1}), for $(b)$ we use equation (\ref{EQ:strong_equation}), for $(c)$ we expand $H(x_{k-1},x_*)^T {\nu}_k$ explicitly, and for $(d)$ we use the property that the $i$-th coordinate of $\nu_k$ is nonnegative if $L_i>0$ (c.f. Assumption \ref{AS:convex}) and the coordinatewise convexity of $F$.

Let us consider beforehand the stationary case.  If there {exists}
a positive integer $k_0$  and a subgradient $ \nu\in \mathcal{V}(x_{k_0-1},x_{k_0})$ such that $\bm{L}^T\nu  =0$, one deduces from  equation~(\ref{EQ:bound_decreasing}) that $g(F(x_{k_0})) = \mathrm{inf}\{g\circ F\}$. Recalling that the sequence $\big(g(F(x_k))\big)_{k\in\mathbb{N}}$ is nonincreasing (cf. Remark \ref{RE:nonincreasing}), one thus has $x_{k_0+j}\in S$ for any $j\in\mathbb{N}$. Using Lemma \ref{LE:sub_property}, it follows that for any $j \in \mathbb{N}$, $x_{k_0+j} \in p(x_{k_0+j})$. Hence, for all $j \in \NN$, $x_{k_0+j} = x_{k_0} \in S$ and the algorithm actually stops at a global minimizer.

This ensures that if there exists $k_0$ such that $1\leq k_0 \leq k$ and a subgradient $ \nu\in \mathcal{V}(x_{k_0-1},x_{k_0})$ with $\bm{L}^T\nu  =0$, 
then equation (\ref{EQ:complexity_f}) holds since in this case, $x_k =x_{k_0} \in S$.

To proceed, we now suppose  that $\bm{L}^T \nu >0$ for every $\nu\in \mathcal{V}(x_{j-1},x_j)$ and for every $j$ in $\{1,\ldots,{k}\}$. Observe first that by (\ref{EQ:bound_decreasing})   the sequence $\|x_k-x_*\|$ is nonincreasing and, since $x_k$ is a descent sequence for $g\circ F$, it evolves within $B(x_*,\|x_0-x_*\|)$ and satisfies $g(F(x_k)) \leq g(F(x_0))$ for all $k \in \NN$. Recalling Lemma \ref{LE:bounded1}, we have the following 
{boundedness} result
\begin{equation}\label{maj}
\CC_{g(F(x_0))}(B_{x_0,x_*})\geq \underset{i=1,\ldots,k}{\mathrm{max}} \big\{ \underset{\nu\in\mathcal{V}(x_{i-1},x_i)}{\mathrm{min}}\{\bm{L}^T {\nu}\} \big\}, \forall k\geq 1.
\end{equation}
The rest of the proof is quite standard. Combining inequalities of the form (\ref{EQ:bound_decreasing}) with the above inequality \eqref{maj} one obtains  
\begin{eqnarray}
g(F(x_j)) - g (F(x_*)) \leq \frac{\CC_{g(F(x_0))}(B_{x_0,x_*}) }{2}(\|x_{j-1} - x_*\|^2 - \|x_j - x_*\|^2), 
\label{EQ:unifBound}
\end{eqnarray}
 for all {$j\in\{1,\ldots,k\}$}
  and for any $x_*\in S$.

Fix $k\geq 1$. Summing up inequality (\ref{EQ:unifBound}) for $j \in \{1, \ldots,k\}$ ensures that for any $x_* \in S$, we have
\begin{eqnarray}
\sum_{j=1}^{k} [g \circ F(x_{j}) - g\circ F(x_{*})] &\leq&\frac{\CC_{g(F(x_0))}(B_{x_0,x_*})}{2}(\|x_0 - x_*\|^2 - \|x_k - x_*\|^2)\nonumber\\
& \leq & \frac{\CC_{g(F(x_0))}(B_{x_0,x_*}) }{2}\|x_0 - x_*\|^2.
\label{EQ:complexity_9}
\end{eqnarray}
Since the sequence $(g\circ F(x_j))_{j\in\mathbb{N}}$ is nonincreasing, it follows that, for any $x_* \in S$,
\begin{align}
	k[g \circ F(x_{k}) - g\circ F(x_{*}) ] \leq \sum_{j=1}^k [g \circ F(x_{j}) - g\circ F(x_{*}) ].
	\label{EQ:boundSum}
\end{align}
Substituting inequality (\ref{EQ:boundSum}) into inequality (\ref{EQ:complexity_9}), we obtain
\begin{eqnarray}
k[g \circ F(x_{k}) - g\circ F(x_{*}) ]  \leq \frac{\CC_{g(F(x_0))}(B_{x_0,x_*})}{2} \|x_0 - x_*\|^2,\ \forall x_*\in S.\nonumber
\end{eqnarray}
Dividing both sides of this inequality by $k$ indicates that (\ref{EQ:complexity_f}) holds. Since $k$ was an arbitrary positive integer, this completes the proof of (i).

\smallskip

\noindent
(ii) The proof relies on Opial's lemma/monotonicity techniques \`a la F\'ejer (see \cite{BC17}). Using \eqref{EQ:complexity_f} and the lower semicontinuity of $g\circ F$ one immediately proves that cluster points of $(x_k)_{k\in \NN}$ are in $S$. Combining this with the fact that $\|x_k-x_*\|$ is nonincreasing for all $x_* \in S$ concludes the proof.\end{proof}

\begin{remark}\label{r}{\rm (a) If the function $g$ is globally $L_g$ Lipschitz continuous, one has $\CC_\gamma(K) \leq L_g \|\bm{L}\|$ for any compact set $K\subset\mathbb{R}^n$ and any $\gamma \geq \min g \circ F$, see Subsection \ref{SE:LIP}.\\
(b) Consider a minimization problem $\PP$ which has several formulations in the sense that there exist $g_1,F_1$ and $g_2,F_2$  such that the objective is given by $g_1\circ F_1=g_2\circ F_2$. Then the complexity results for the two formulations may differ considerably, see Subsection~\ref{SE:FB}. \\
(c) Note that the above proof actually yields a more subtle ``online" {estimate}:
\begin{equation}\label{EQ:complexity_f2}
g (F(x_k)) - g(F(x_*)) \leq \frac{\underset{i=1,\ldots,k} {\mathrm{max}} \big\{ \min\left\{\bm{L}^T {\nu}: \nu\in\mathcal{V}(x_{i-1},x_i) \big\}\right\}}{2k} \|x_0 - x_*\|^2.
\end{equation}
This shows that the specific history of a sequence plays an important role in its actual complexity. This is of course not captured by global constants of the form $\CC_{g(F(x_0))}(B_{x_0,x_*})$ which are worst case estimates.\\
(d) The complexity estimate of Theorem \ref{TH:nonincreasing} does not directly involve the constant $L_g$, {but} only multipliers. This will be useful to recover existing complexity results for {algorithms} such as the forward-backward splitting algorithm in Section \ref{SE:FB}.
}
\end{remark}

\paragraph{Explicit complexity bounds}
We now provide an explicit bound for the condition number which will in turn provide explicit complexity bounds for \alg. Our approach  relies on a thorough study of the multipliers and on a measure of the Slater's like assumption through the term $$\mathrm{dist}[F(\bar{x}),\mathrm{bd}\ \dom g]>0,$$
whose positivity follows from Assumption~\ref{AS:Qualification}.
Our results on multipliers are recorded in the following fundamental lemma. Its proof is quite delicate and  it is postponed 
in the appendix. 

\smallskip

Given a matrix $A\in \RR^{m\times n}$, its operator norm is denoted by $\|A\|_{\rm op}$.
\begin{lemma}[Bounds for the multipliers of {$\PPx$}\,]
For any $x\in F^{-1}(\dom g )$, $y\in p(x)$ and $\nu\in\mathcal{V}(x,y)$, the following statements hold:\\
 (i) if $H(x,y)\in \mathrm{int}\ \dom g$, then $\bm{L}^T\nu \leq L_g \|\bm{L}\|$;\\
 (ii) if $H(x,y)\in \mathrm{bd}\ \dom g$, then \\
%
%
$\displaystyle				\bm{L}^T\nu 
\leq   8L_g \left(\|\nabla F(x)\|_{\rm op} +(3\|x - y\| + \|\bar{x} - x\|) \frac{\|\bm{L}\|}{2} \right)^2 
\frac{ \mathrm{dist}[F(\bar{x}),\mathrm{bd}\ \dom g ] + \|\bm{L}\| \|\bar{x} - x \|^2/2 } {\mathrm{dist}[F(\bar{x}),\mathrm{bd}\ \dom g ]^2 }$
$\quad$ where $\bar{x}$ is as in Assumption \ref{AS:Qualification}.

\label{TH:explicitBound}
\end{lemma}
{The full proof of this Lemma  is postponed to Appendix \ref{SE:appendixExplicitBound}.} A pretty direct consequence of 
{Lemma \ref{TH:explicitBound}}
 is a complexity result with explicit constants.

\begin{theorem}[Explicit complexity bound for \alg] \label{CO:explicitLipDomain}
Let $(x_k)_{k\in\mathbb{N}}$ be a sequence generated by \alg . Then, for any $x_*\in S$ {and} for all $k \geq 1$,  
\begin{equation}
{g (F(x_k)) - g(F(x_*)) \leq  \max\Big\{\|\bm{L}\|,\gamma \Big\}\, \frac{L_g\|x_0 - x_*\|^2}{2k},}
\label{EQ:complexity_f_explicit}
\end{equation}
where
 $$
	\gamma=8 \left(\|\nabla F(x_0)\|_{\rm op} + \frac{\|\bm{L}\|}{2}\left( 11\|x_0 - x_*\| + \|\bar{x} - x_*\| \right) \right)^2 
				\frac{ d_{\bar{x}} + \frac{\|\bm{L}\|}{2}\left( \| x_* - \bar{x}\|+ \| x_* - x_0\| \right)^2 } {d_{\bar{x}}^2 }
$$
{with} $d_{\bar{x}} = \mathrm{dist}[F(\bar{x}),\mathrm{bd}\ \dom g ]$. 
\end{theorem}
\begin{proof}
%
			
			Fix $k>0$ in $\NN$. {Recall that $\|x_k - x_*\| \leq \|x_0 - x_*\|$ (see the proof of Theorem \ref{TH:nonincreasing})}. 
				The bound on the operator norm of the Jacobian is then computed as follows:
\begin{align*}
				\|\nabla F(x_k)\|_{\rm op} &= \|\nabla F(x_0) + \nabla F(x_k) - \nabla F(x_0)  \|_{\rm op} \\
				&\leq \|\nabla F(x_0)\|_{\rm op}+ \|\nabla F(x_k) - \nabla F(x_0)\|_{\rm op} \\
				&\leq \|\nabla F(x_0)\|_{\rm op}+ \|\bm{L}\|\|x_k - x_0\| \\
				&\leq \|\nabla F(x_0)\|_{\rm op}+ 2\|\bm{L}\|\|x_* - x_0\|,
\end{align*}
where we have used the fact that row $i$ of $\nabla F$ is $L_i$ Lipschitz, for $i = 1, \ldots, m$, implying that $\nabla F$ is $\|\bm{L}\|$ Lipschitz with respect to the Frobenius norm. One concludes by using Theorem \ref{TH:nonincreasing} and Lemma \ref{TH:explicitBound}.\end{proof}

\subsection{Consequences of the main result}
\label{SE:conseqComplexity}


\subsubsection{Complexity for Lipschitz continuous models} 
\label{SE:LIP}
It is very useful to make the following elementary observation:
\begin{align}
	\label{bound}
	\CC_\gamma(K)&\leq \sup \{\bm{L}^T \nu:\: x\in K,\, g(F(x)) \leq \gamma,\, y\in p(x),\nu\in \mathcal{V}(x,y) \}\nonumber\\
& { \leq  \sup \{\bm{L}^T \lambda:\: x\in K, \, g(F(x)) \leq \gamma, \, y\in p(x),\lambda\in \partial g({H}(x,y)) \}} \nonumber\\
&{\leq  \sup \{\bm{L}^T \lambda:\:  z\in \mathrm{dom}\ g,\lambda\in \partial g(z) \} ,}
\end{align}
{where the first inequality is because we just removed an infimum from the definition of $\CC_\gamma(K)$ in 
 {Lemma \ref{LE:bounded1}}, the second follows because for any $x \in F^{-1}(\dom g)$ and $y \in p(x)$, $\mathcal{V}(x,y) \subset \partial g\left( H(x,y) \right)$, and the third follows because for any $x\in F^{-1}(\dom g)$ and any $y\in p(x)$, we have $H(x,y) \in \dom  g$. The above upper bound is finite whenever $g$ has full domain and is globally Lipschitz continuous.}
Indeed, in that case $\sup_{z,\lambda}\{ \|\lambda\|:\lambda \in \partial g(z)\}\leq L_g$ (see e.g.,  \cite[Theorem 9.13]{Rockafellar:1998}). 
Assumptions~\ref{AS:existence_sub} and \ref{AS:Qualification} are automatically satisfied. An immediate application of the Cauchy-Schwartz inequality leads to the following bound for the condition number
\begin{equation}
\CC_{g(F(x_0))}(B_{x_0,x_*})\leq L_g\|\bm{L}\|, \quad \text{for all feasible data\, $x_0,x_*$}.
\label{cond}
\end{equation}
Thus we have the general result for composite Lipschitz continuous problems:
\begin{corollary}[Global complexity for Lipschitz continuous model]
 \label{CO:bounded}
	In addition to Assumptions~\ref{AS:convex} and \ref{AS:effectiveness}, suppose that $g$ is finite valued, and thus globally Lipschitz continuous with Lipschitz constant $L_g$. Let $(x_k)_{k\in\mathbb{N}}$ be a sequence generated by \alg, then  it converges to a minimizer and for any $x_* \in S$:
	\begin{align}
		g (F(x_k)) - g(F(x_*)) &\leq \frac{L_g\|\bm{L}\|}{2k} \|x_0 - x_*\|^2,\quad \forall k\geq1.
		\label{EQ:complexLipG}
	\end{align}	
\end{corollary}

\begin{remark}{\rm
(a) Note that instead of using the Cauchy-Schwartz inequality, one could use H\"older's inequality if $g$ is Lipschitz with respect to a different norm. For example, suppose that each coordinate of $g$ is $L_g$ Lipschitz continuous (the others being fixed), or in other words that the supremum norm of the subgradients of $g$ is bounded by $L_g$. In this case,  a result similar to (\ref{EQ:complexLipG}) holds with $\|\bm{L}\|_1$ in place of the Euclidean norm. \\
(b) The bound given above is   sharper than the general bound provided in Theorem~\ref{CO:explicitLipDomain}.}
\end{remark}

\subsubsection{Proximal Gauss-Newton's method for 
{min-max} problems}
\label{SE:MINMAX}

Let us illustrate how Theorem \ref{TH:nonincreasing} can give new insights into the proximal Gauss-Newton method (PGNM) when $g$ is a componentwise maximum. As in Subsection~\ref{minmax} consider $g(z) =\max(z_1,\ldots,z_m)\label{EQ:maxG}$ for $z$ in $\RR^m$. Take $F$ as in Assumption~\ref{AS:convex}, set  $\displaystyle L = \max_{i \in \{1, \ldots, m\}} L_i $ and
\begin{align}
	\bm{L} = (L,\ldots,L)^T.
	\label{EQ:GNmodelL}
\end{align}
\alg\ writes:
\begin{align}
	x_{k+1} =\underset{y \in \RR^n}{ \argmin} \max_{i \in \left\{ 1,\ldots,m \right\}} \left\{f_i(x_k) + \nabla f_i(x_k)^T( y - x_k) \right\} + \frac{L}{2}\|y - x_k\|^2,
	\label{EQ:GNAlgo}
\end{align}
which is nothing else than PGNM applied to the problem $\PP$. 

The kernel $g$ is $1$ Lipschitz continuous with respect to the $L^1$ norm. As in Corollary~\ref{CO:bounded}, a straightforward application of H\"older's inequality leads to the following complexity result:
\begin{corollary}[Complexity for PGNM]
	\label{CO:GNMAX}
	In addition to Assumptions \ref{AS:convex}, \ref{AS:effectiveness}, suppose that $g$ is the componentwise maximum. Let $(x_k)_{k\in\mathbb{N}}$ be a sequence generated by PGNM, then for any $k\geq 1$ and any $x_*$ in $S$, we have 
	\begin{align}
		g (F(x_k)) - g(F(x_*)) &\leq \frac{L}{2k} \|x_0 - x_*\|^2,
		\label{EQ:complexLipG}
	\end{align}
	furthermore, the sequence $(x_k)_{k \in \NN}$ converges to a solution of 
	 $\PP$.
\end{corollary}
\begin{remark}\label{RE:GNMAX} {\rm (a) As far as we know, this complexity result for the classical PGNM is new. We suspect that similar results could be derived for much more general kernels $g$, this is a matter for future research.\\
(b)  Note that the accelerated algorithm for PGNM described in \cite{Drusvyatskiy:20162} would achieve a convergence rate of the form $\displaystyle 2\sqrt{m}L\|x_0-x_*\|^2/k^2$. Indeed, the multiplicative constant appearing in the convergence rate of \cite[Theorem 8.5]{Drusvyatskiy:20162} involves the Lipschitz constant of $\nabla F^T$ measured in term of operator norm. We do not know if the constant $\sqrt{m}$ (which can be big for some problems) could be avoided, and thus, at this stage of our understanding, we cannot draw any comparative conclusion between these two complexity results.\\
(c) 
 Although the worst-case complexity estimate of \alg\  is the same as the one for PGNM, we have observed a dramatic difference in practice, see Section~\ref{SE:numerical}. The intuitive reason is quite obvious since \alg\ is much more adapted to the geometry  of the problem. Better performances might also be connected to the estimate 
  \eqref{EQ:complexity_f2} given in Remark~\ref{r}. 
 }
\end{remark}

\subsubsection{Complexity of the Moving balls method.}
\label{SE:MB}
We provide here 
{an} enhanced nonsmooth version of the moving balls method, introduced in Subsection~\ref{movi}, which allows to handle sparsity constraints. Consider the  following  nonlinear convex programming problem:
\begin{align}
       \min_{x \in \RR^n} \,&f(x)+h(x)\nonumber\\
 \mathrm{s.t.}\,&f_i(x) \leq 0, \, i=
 1, \ldots, m.
	\label{EQ:modelNLP2}
\end{align}
where {$f:\RR^n\rightarrow \RR$ is convex, differentiable with $L_f$ Lipschitz continuous gradient}, each $f_i{:\RR^n\rightarrow \RR}$ is defined as in Assumption \ref{AS:convex}, and $h \colon \RR^n \to \RR$ is a convex lower semicon\-ti\-nuous function, for instance $h=\|\cdot\|_1.$


Choosing $g$ and $F$ adequately (details can be found in the proof of Corollary \ref{CO:MovingBall}), \alg\ gives an algorithm combining/improving ideas presented in \cite{Auslender:10,shefi2016dual}\footnote{Observe that the subproblems are simple convex quadratic problems.}:
\begin{align}
	x_{k+1} \in {\underset{y\in\RR^n}{\mathrm{argmin}}}\quad & f(x_k) + 
	{\nabla f(x_k)^T( y - x_k)}
	 + \frac{L_f}{2}\|y - x_k\|^2 + h(y) \nonumber \\
	\mathrm{s.t.}\quad&f_i(x_k) + 
	{\nabla f_i(x_k)^T( y - x_k)}
	+ \frac{L_i}{2}\|y - x_k\|^2 \leq 0,\, i=1,\ldots,m.
	\label{EQ:MBalgo}
\end{align}

Our main convergence result in Theorem \ref{TH:nonincreasing} can be combined with Lemma \ref{LE:bounded1} to recover and extend the convergence results of \cite{Auslender:10,shefi2016dual}. More importantly we derive explicit complexity bounds of the form $O(\frac{1}{k})$. We are not aware of any such quantitative result for general nonlinear programming problems.
\begin{corollary}[Complexity of the moving balls method]
	\label{CO:MovingBall}
	Assume that $h$ is $L_h$ Lipschitz continuous and that there exists $\bar{x}$ in $\RR^n$ such that $f_i(\bar{x}) < 0$, $i = 1, \ldots, m$. Assume that (\ref{EQ:modelNLP2}) has a solution $x^*$ and that there exists $i$ in $\{1,\ldots,m\}$ such that $L_i > 0$. Let $(x_k)_{k \in \NN}$ be 
	a sequence generated by \alg. Then for any $k \geq 1$, $x_k$ is feasible\footnote{For the original problem (\ref{EQ:modelNLP2})} and 
\begin{equation}
				(f+h)(x_k) - (f+h)(x_*) \leq  \max\Big\{1,\zeta\Big\} \frac{\|\bm{L}\|(1+ L_h)\|x_0 - x_*\|^2}{2k},
\label{EQ:complexity_MB_explicit}
\end{equation}
where {$\bm{L} = (L_f,L_1,\ldots,L_m,0,\ldots,0)^T\in\RR^{n+m+1}$ and }
 $$\zeta=8\|\bm{L}\| \left(\frac{\|\nabla F(x_0)\|_{\rm op}}{\|\bm{L}\|} + \frac{1}{2}\left( 11\|x_0 - x_*\| + \|\bar{x} - x_*\| \right) \right)^2 \\
				 \frac{ \bar{f} + \frac{\|\bm{L}\|}{2}  \left( \| x_* - \bar{x}\|+ \| x_* - x_0\| \right)^2 } {\bar{f}^2 }$$
with $\bar{f} = |\max_{i=1, \ldots, m}\left\{ f_i(\bar{x}) \right\}|$. 
\end{corollary}
\begin{proof}
				We set $F \colon x \mapsto (f(x),f_1(x), \ldots, f_m(x),x)$ and $g\colon (z_0,z_1, \ldots, z_m, z) \mapsto z_0 + h(z) + \sum_{i=1}^m i_{\RR_-}(z_i)$. With this choice, we obtain problem (\ref{EQ:modelNLP2}) and algorithm (\ref{EQ:MBalgo}) (we set the smoothness modulus of the identity part in $F$ to $0$, whence the value of ${\bm L}$). Assumptions \ref{AS:convex}, \ref{AS:effectiveness} and \ref{AS:Qualification} are clearly satisfied. Assumption \ref{AS:existence_sub} is satisfied since one of the $L_i$ is positive, {indicating} that the subproblems in (\ref{EQ:MBalgo}) are  strongly convex  {and} have bounded constraint sets. Finally $g$ is $(1+L_h)$ Lipschitz continuous on its domain. Hence 
				{Theorem} \ref{CO:explicitLipDomain} can be applied. It remains to notice that $\bar{f} = \mathrm{dist}[F(\bar{x}), \mathrm{bd}\ \mathrm{dom}\ g]$ to conclude the proof.
\end{proof}

\subsubsection{Forward-backward splitting algorithm}
\label{SE:FB}
To illustrate further the flexibility of our method, we explain how our approach allows to recover the classical complexity results of the classical forward-backward splitting algorithm {within the convex setting}. 
 Let $f\colon \mathbb{R}^n \to \mathbb{R}$ be a continuously differentiable convex function with $L$ Lipschitz gradient and $h\colon \mathbb{R}^n \to \mathbb{R} \cup \{+\infty\}$ be a proper lower semicontinuous convex function. Consider the following problem
\begin{align}
	\inf_{x \in \mathbb{R}^n} f(x) + h(x).
	\label{EQ:smoothPlusNonSmooth}
\end{align}
This problem is a special case of the optimization objective of problem $\PP$, by choosing

\begin{equation}F:\left\{\begin{array}{lcl}
	\RR^n &\rightarrow & \RR^{n+1} \nonumber \\
	x & \rightarrow  &(f(x),x)
	\label{EQ:FBmodelF}
\end{array}\right.
\end{equation}
and 
\begin{equation}
	\label{EQ:FBmodelg}
	g:
	\left\{
		\begin{array}{lcl}
			\RR\times \RR^{n} &\rightarrow & 
			{\mathbb{R}\cup \{+\infty\}} \\
			(a,z) & \rightarrow  &a+h(z)
		\end{array}\right.
\end{equation}
%
%
Finally, setting
\begin{align}
	\bm{L} = \left( 
	\begin{array}{c}
		L\\
		0\\
		\vdots\\
		0
	\end{array}
	\right) \in \RR^{n+1},
	\label{EQ:FBmodelL}
\end{align}
\alg\ eventually writes:
\begin{align}
	x_{k+1} = {\underset{{y \in \RR^n}}{\argmin}}\quad f(x_k) +
	 {\nabla f(x_k)^T( y - x_k)}
	 + h(y) + \frac{L}{2}\|y - x_k\|^2,
	\label{EQ:FBAlgo}
\end{align}
which is exactly the forward-backward splitting algorithm. It is immediate to check that Assumptions \ref{AS:convex}, \ref{AS:effectiveness}, \ref{AS:existence_sub} and \ref{AS:Qualification} hold true as long as the minimum is achieved in (\ref{EQ:smoothPlusNonSmooth}) and that $\dom h$ has nonempty interior with $h$ being Lipschitz continuous on $\dom h$\footnote{Lipschitz continuity is actually superfluous for Theorem \ref{TH:nonincreasing} to hold. }. Given the form of $g$ in equation (\ref{EQ:FBmodelg}) and $\bm{L}$ in (\ref{EQ:FBmodelL}), Theorem~\ref{TH:nonincreasing} yields the classical convergence and complexity results for the forward-backward algorithm (see e.g.,  \cite{Combettes:2005} for convergence and \cite{Beck:2009} for complexity): $x_k$ converges to a minimizer and 
\begin{align}
		(f+h)(x_k) -(f+h)(x_*) \leq \frac{L \|x_0 - x_*\|^2}{2k}, \quad \forall k \geq 1.		\label{EQ:complexFB}
\end{align}


\begin{remark}[Complexity estimates depend on the formulation] {\rm Assume that $h$ is the indicator function of a ball $B(a,r)$ so that the above method is the gradient projection method on this ball and its complexity is recovered by equation (\ref{EQ:complexFB}). Another way of modeling the problem is to consider minimizing $g_1\circ F_1$ taking the following forms
  \begin{align}
 & F_1\, \colon \mathbb{R}^n\rightarrow \mathbb{R}^2\nonumber\\
&  \ \ \ \ \ x\rightarrow \left(
 \begin{array}{c}
 f(x)\\
 f_1(x)
 \end{array}\right),\nonumber
 \end{align}
 where $f_1(x) =   \|x-a\|^2 - r^2$, and
 \begin{align}
 & g_1\, \colon \mathbb{R}^2 \rightarrow \mathbb{R}{\ \cup\ \{+\infty\}}\nonumber\\
 &\ \ \ \ \ \left(\begin{array}{c}
 z_1\\
 z_2
 \end{array}\right)\rightarrow z_1 + h_1(z_2),\nonumber
 \end{align}
 where $h_1$ is the indicator function of $\mathbb{R_-}$, so that for every $x\in\mathbb{R}^n$ it holds $g\circ F(x) = g_1\circ F_1(x)$ and \alg\ for $g_1\circ F_1$ is equivalent to the moving balls method.  Since 
 $$f_1(x) + 
 {\nabla f_1(x)^T(y-x)}
  + \frac{L_1}{2}\|y - x\|^2 = f_1(y),\ \forall  x,y\in\mathbb{R}^n,$$
 where $L_1=2$ is the Lipschitz constant of the gradient of $f_1$, it follows that for any $x\in\mathbb{R}^n$ it holds
 $$\underset{\text{gradient projection method}}{\underbrace {\underset{y\in\mathbb{R}^n}{{\mathrm{argmin}}} \bigg(\frac{1}{2} \left\| y - x + \frac{\nabla f(x)}{L}\right\|^2 + h(y)\bigg ) }}= \underset{\text{moving balls method}}{\underbrace { \underset{y\in\mathbb{R}^n}{{\mathrm{argmin}}}  \bigg(\frac{1}{2} \left\| y - x + \frac{\nabla f(x)}{L}\right\|^2 + h_1(f_1(y))\bigg) }}.$$
Thus, if the initial point is the same, the sequence of the moving balls method is the same as that of gradient projection method. However, considering the third item of Remark \ref{r}, the best estimate our analysis can provide for the moving-balls method is
 \begin{eqnarray}
	\label{EQ:complexMBBis}
	& & (f+h_1)(x_k) - (f+h_1)(x_*)  \\
	& \leq & {\frac{\underset{i=1,\ldots,k}{\mathrm{max}} \big\{L + 2 \mathrm{min}\{  \nu_1\in\mathbb{R}:  (1,\nu_1) \in \mathcal{V}(x_{i-1},x_i)\}\big\}}{2 k} \|x_0 - x_*\|^2,\ \forall k \geq 1.}\nonumber
 \end{eqnarray}
 This is different from the complexity of the gradient projection method in (\ref{EQ:complexFB}). Indeed the {infimum over} the variable $\nu_1$ 
 appearing in the numerator is nonnegative. {Furthermore it is non zero in many situations because otherwise the constraints would never been 
 binding. As an example, one can consider a linear objective function for which the infimum in the numerator in (\ref{EQ:complexMBBis}) is 
 strictly positive}. 
 
Observe that the numerator in (\ref{EQ:complexMBBis}) is strictly greater than $L$ which is the classical constant attached  the projected gradient, even though both algorithms are actually the same. This highlights the second item of Remark \ref{r} on the dependance of the estimate on the choice of equivalent composite models. 
}
\end{remark}

\section{Backtracking and linesearch}\label{s:search}
In practice, the  collection $\bm{L}$ of Lipschitz constants  may not be known or  efficiently computable. Lipschitz continuity might not even be  global. To handle these fundamental cases, we provide now our algorithmic scheme in (\ref{Numerical_Scheme}) with a  linesearch procedure (see e.g., \cite{nocedal,Beck:2009}).

First let us define a space search for our steps\footnote{Actually the inverse of our steps.}
\begin{eqnarray}
\Gamma:=\Big\{ ({\alpha}_1,\ldots,{\alpha}_m)\in\mathbb{R}^m_+ :\: {\alpha}_i
= 0\ \text{if} \ L_i = 0,\ {\alpha_i > 0\ \text{if}\ L_i > 0},\ i=1,\ldots,m\Big\}, \nonumber
\end{eqnarray}
and for every $\bm{\alpha}\in\Gamma$ we
set $$H_{{\bm{\alpha}}}(x,y):= F(x) + \nabla F(x)(y-x) + \frac{{\bm{\alpha}}}{2}\|y - x\|^2,\ \forall (x,y)\in\mathbb{R}^n\times\mathbb{R}^n.$$
In order to design an algorithm with this larger family of surrogates, we need  a stronger version of Assumption \ref{AS:existence_sub}.
\begin{assumption}
	\label{AS:existence_sub3}
For any $x\in F^{-1}(\dom g)$ and every ${\bm{\alpha}}\in\Gamma$, the function $g\circ H_{{\bm{\alpha}}}(x,\cdot)$ has a minimizer.
\end{assumption}
\noindent For every $x\in F^{-1}(\dom g)$, the basic subproblem we shall use is defined for any ${\bm{\alpha}}\in\Gamma$,
\begin{eqnarray}
\big(\mathcal{P}_{{\bm{\alpha}},x}\big):\ \ \ \ p_{{\bm{\alpha}}}(x):= \underset{y\in\mathbb{R}^n}{\mathrm{argmin}}\ \ g\circ H_{{\bm{\alpha}}}(x,y).\nonumber
\end{eqnarray}
The \alg\ algorithm with backtracking 
{step sizes}
 is defined as:
\newpage

 \begin{mdframed}[style=MyFrame]
 \begin{center}
 {\bf Multiproximal method with backtracking 
{step sizes}
 }
 \end{center}
 \bigskip
 \noindent Take $\ x_0 \in F^{-1}(\dom g),\ \bm{\alpha}_0\in \Gamma\ \text{and}\ \eta>1$. Then iterate for $k \in \NN$:
\begin{align}
\label{Numerical_Scheme_Back}
\left.
	\begin{array}{ll}
		step\ 1.&\text{set } \tilde{\bm{\alpha}}=\bm{\alpha}_k,\, \tilde{x} \in p_{\bm{\tilde{\alpha}}}(x_k).\\
		step\ 2.& \text{while } {\text{the inequality}\ H_{\tilde{\bm{\alpha}}}(x_k, \tilde{x}) >  F(\tilde{x})\ \text{is not satisfied}}:\\
		&
		\begin{array}{lll}
			&\text{for}&i=1,\ldots,m:\\
			&&\text{if } f_i(\tilde{x}) > f_i(x_k) + 
			{\nabla f_i(x_k)^T(\tilde{x} - x_k)}
			+ \frac{
			{\tilde{\alpha}_i}}{2}\|\tilde{x} - x_k\|^2:\\
			&&\quad\text{set } 
			{{\tilde{\alpha}}_i \leftarrow \eta {\tilde{\alpha}}_i}\\
			&\text{set }&\tilde{x} \in p_{\bm{\tilde{\alpha}}}(x_k).
		\end{array}\\
		&\text{set }	\bm{\alpha}_{k+1}=\tilde{\bm{\alpha}}.\\
	 	step \ 3.&\text{if}\ x_{k}\in p_{\bm{\alpha}_{k+1}}(x_{k}):\\
	 	&\quad x_{k+1} = x_{k}\\
	 	&\text{else :}\\
		&\quad x_{k+1}=\tilde{x}
	\end{array}
\right\rbrace
\end{align}
\end{mdframed}

\begin{remark}[A finite while-loop]\label{RE:backtracking_Lk}{\rm (a) One needs to make sure that the scheme is well defined and that each while-loop stops after finitely many tries. Under Assumption \ref{AS:convex},  it is naturally the case. 
To see this let $(\bm{\alpha}_k)_{k\in\mathbb{N}}$ be a sequence generated by the scheme in (\ref{Numerical_Scheme_Back}), and let $\bm{L}\in\mathbb{R}^m_+$ be the collection of Lipschitz constants associated to $\nabla F$. Then, for every integer $k$, the following statements must obviously hold for each $i=1,\ldots,m$:
	$$(\bm{\alpha}_{k})_i\leq \max\left\{\eta L_i, (\bm{\alpha}_{0})_i  \right\}.$$
(b) More importantly we shall also see that local Lipschitz continuity and coercivity also ensure that the while-loop is finite, see Theorem~\ref{t:gen} below.}
\end{remark}

\noindent Arguments similar to those of Subsection~\ref{s:cr} allow to derive  chain rules and to eventually consider the following sets (note that {a proposition equivalent to Proposition \ref{PR:qualifSub}} holds for the problem $\big(\mathcal{P}_{{\bm{\alpha}},x}\big)$).
\begin{definition}[Lagrange multipliers of the subproblems]
Given any fixed point $x\in F^{-1}(\dom g)$, any ${\bm{\alpha}}\in\Gamma$ and any $y\in p_{{\bm{\alpha}}}(x)$, we set
$$\mathcal{V}_{{\bm{\alpha}}}(x,y): = \{\nu\in g(H_{{\bm{\alpha}}}(x,y)): 
{\nabla _y H_{{\bm{\alpha}}}(x,y)^T}
 \nu = 0\}.$$
\end{definition}
\noindent We are now able to extend Theorem \ref{TH:nonincreasing} to a larger setting:  Lipschitz constants do exist but they are unknown to the user.
\begin{theorem}[\alg\ with backtracking]
	\label{TH:nonincreasingBack}
	Suppose that 
	{Assumptions} \ref{AS:convex}, \ref{AS:effectiveness}, \ref{AS:Qualification} and \ref{AS:existence_sub3} hold.
	Let $(x_k)_{k\in\mathbb{N}}$ and $(\bm{\alpha}_k)_{k\in\mathbb{N}}$ be any sequences generated by the algorithmic scheme in (\ref{Numerical_Scheme_Back}). Then, for every $x_*\in S$ and any sequence $(\nu_k)_{k\in\note{\mathbb{N}}}$ such that $\nu_j\in \mathcal{V}_{\bm{\alpha}_j}(x_{j-1},x_j)$ for all $j\geq 1$, one has,
	\begin{eqnarray}
	g\circ F(x_k) - g\circ F(x_*) \leq  \frac{ \underset{j=1,\ldots,k}{\mathrm{max}} \{ \hat{\bm{\alpha}}^T \nu_j\}}{2k}\|x_0 - x_*\|^2,\ \forall k\geq 1,
	\label{EQ:complexity_back}
	\end{eqnarray}
	where $\hat{\bm{\alpha}} \in \Gamma$ is the vector 
	{whose} entries are given by the upper bound in Remark \ref{RE:backtracking_Lk}. Furthermore, the sequence $(x_k)_{k\in \NN}$ converges to a point in S.
\end{theorem}
\begin{proof}
The proof is in the line of that of Theorem \ref{TH:nonincreasing}. We observe first that we have a descent method.	Let $(x_k)_{k\in\mathbb{N}}$ and $(\bm{\alpha}_k)_{k\in\mathbb{N}}$ be sequences generated by the algorithmic scheme in (\ref{Numerical_Scheme_Back}). Fix $k \geq 1$. One has $F(x_k) \leq H_{\bm{\alpha}_k}(x_{k-1},x_k)$ as the while-loop stops after finitely many steps (see Remark \ref{RE:backtracking_Lk}). Using the monotonicity properties of $g$, one deduces that $g(F(x_k)) \leq g(H_{\bm{\alpha}_k}(x_{k-1},x_k))$. Considering that $x_k$ is a minimizer of $g(H_{\bm{\alpha}_k}(x_{k-1},\cdot))$, it follows further that $g(H_{\bm{\alpha}_k}(x_{k-1},x_k)) \leq g( H_{\bm{\alpha}_k}(x_{k-1},x_{k-1}) )= g( F(x_{k-1}))$, indicating that the algorithm is a descent method.

For any $\nu_k$ in $\mathcal{V}_{\bm{\alpha}_k}(x_{k-1},x_k)$ and any $x_*$ in $S$, one has
\begin{eqnarray}
&& g\circ F(x_k) - g\circ F(x_*)\nonumber\\
 &\overset{(a)}{\leq}& g\circ H_{\bm{\alpha}_k}(x_{k-1},x_k) - g\circ F(x_*)\nonumber\\
&\overset{(b)}{\leq} &[H_{\bm{\alpha}_k}(x_{k-1},x_k) - F(x_*) ]^T \nu_k\nonumber\\
&\overset{(c)}{=}& H_{\bm{\alpha}_k}(x_{k-1},
{x_*})^T \nu_k - F(x_*)^T \nu_k - \frac{\bm{\alpha}_k^T \nu_k}{2} \|x_k - x_*\|^2\nonumber\\
&\overset{(d)}{=}& [F(x_{k-1}) + \nabla F(x_{k-1})(x_* - x_{k-1})]^T \nu_k - F(x_*)^T \nu_k + \frac{\bm{\alpha}_k^T \nu_k}{2} (\|x_{k-1} - x_*\|^2 - \|x_k - x_*\|^2)\nonumber\\
&\overset{(e)}{\leq}&  \frac{\bm{\alpha}_k^T \nu_k}{2} (\|x_{k-1} - x_*\|^2 - \|x_k - x_*\|^2),\nonumber
\end{eqnarray}
where: (a) is obtained by combining the descent property with $F(x_k) \leq H_{\bm{\alpha}_k}(x_{k-1},x_k)$, (b) follows from the convexity of $g$ and the fact that $\nu_k\in \mathcal{V}_{\bm{\alpha}_k}(x_{k-1},x_k)$, (c) is obtained by combining  $[\nabla_y H_{\bm{\alpha}_k}(x_{k-1},x_k)]^T \nu_k = 0$ with equation (\ref{EQ:strong_a}), (d) is an explicit expansion of $H_{\bm{\alpha}_k}(x_{k-1},x_k)^T\nu_k$ explicitly, and eventually, (e) stems from the property that the $i$-th coordinate of $\nu_k$ is nonnegative if $L_i>0$ (cf. Assumption \ref{AS:convex}), the construction of $\Gamma$ and the coordinatewise convexity of $F$.
We therefore obtain:
\begin{eqnarray}
g(F(x_k)) - g(F(x_*)) \leq \frac{C_k}{2k} \|x_0 - x_*\|^2,\ \ k\geq1,
\end{eqnarray}
where $\displaystyle C_k=\max_{j=1,\ldots,k} \min \{ \bm{\alpha}_j{^T} \nu: \nu\in \mathcal{V}_{\bm{\alpha}_j(x_{j-1},x_j) }\}$ is  a bounded sequence.

The convergence of the sequence $(x_k)_{k\in \NN}$ follows by similar arguments as in the proof of Theorem \ref{TH:nonincreasing}.\end{proof}

We now consider the fundamental and ubiquitous case when global Lipschitz constants do not exist but exist locally. A typical case of such a situation is given by problems involving $C^2$ mappings $F$ (local Lipschitz continuity follows indeed from a direct application of the mean value theorem to the $\nabla f_i$, $i=1,\ldots,m$). To compensate for this lack of global Lipschitz continuity, we make the following coercivity assumption:
\begin{assumption}
	\label{AS:existence_sub4}
	
The problem $g\circ F$ is coercive and, for any $x\in F^{-1}(\dom g)$, the function $g\circ H_{{\bm{\alpha_0}}}(x,\cdot)$ is coercive.
\end{assumption}

\begin{theorem}[Convergence without any global Lipschitz condition]\label{t:gen} Suppose that  Assumptions~\ref{AS:effectiveness}, \ref{AS:Qualification} and \ref{AS:existence_sub4} hold and that Assumption \ref{AS:convex} is weakened in the sense that we  assume that each $\nabla f_i$ is only locally Lipschitz continuous.
Then any sequence $(x_k)_{k\in \NN}$ generated by (\ref{Numerical_Scheme_Back}) converges to a single point in the solution set $S$.
\end{theorem}
\begin{proof} The first and main point to be observed is that the while-loop is finite. Fix $k\geq 1$. Observe that since $\tilde \alpha\geq \alpha_0$ any of the sublevel sets of the subproblems  in the while loop are contained  in a common compact set $K:=\{y\in \RR^n:g\circ H_{\bm{\alpha}_0}(x_k,y)\leq g\circ F(x_k))\}$. Since  each gradient is Lipschitz continuous globally on $K$ the while loop ends in finitely many runs. Using arguments similar to those previously given allows to prove that the sequence $x_k$ is a descent sequence. As a consequence it evolves in set $\{y\in \RR^n:g(F(y))\leq g(F(x_0))\}$ which is a compact set by coercivity. Since on this set $\nabla F$ is Lipschitz continuous, usual arguments {apply}. One can thus prove that the sequence converges as previously. 
\end{proof}
\begin{remark} {\rm ``Complexity estimates" could also be derived but the constants appearing in the estimates would be unknown a priori. In a numerical perspective, they could be updated online and used to forecast the efficiency of the method step after step.}
\end{remark}

\section{Numerical experiments}\label{SE:numerical}

In this section, we present numerical experiments which illustrate the performance of the proposed algorithm on some collections of min-max problems. The proposed method is compared to both the Proximal Gauss-Newton method (PGNM) and an accelerated variant \cite[Algorithm 8]{Drusvyatskiy:20162}, which we refer to as APGNM. In this setting the convergence rate for both \alg\ and PGNM is of the order $O(1/k)$, while it is of the order of $O(1/k^2)$ for APGNM (see \cite{Drusvyatskiy:20162} and the discussion in Subsection \ref{SE:MINMAX}).

\subsection{Problem class and algorithms}

We consider the problem of minimizing the maximum of finitely many quadratic convex functions:
\begin{equation}
\begin{array}{rl}
\underset{x\in\mathbb{R}^n}{\mathrm{min}}\ &\mathrm{max}\{f_1(x),\ldots,f_m(x)\},\label{EQ:minmax_p}\\
\text{with}\ &f_i\colon x \to x^T Q_i x + b_i^T x  + c_i \quad \text{for} \ i=1,\ldots,m,
\end{array}
\end{equation}
where  $Q_1, \ldots, Q_m$ are real $n\times n$ positive semidefinite matrices, $b_1$, $\ldots, b_m$ are in $\mathbb{R}^n$, and $c_1,\ldots,c_m$ in $\mathbb{R}$. Choosing $g$ as the coordinatewise maximum and each coordinate of $F$ to be one of the $f_i$, $i=1,\ldots,m$, one sees that this problem is of the form of $\PP$. Assumptions~\ref{AS:convex} and~\ref{AS:Qualification} are satisfied and we  assume that the data of the problem (\ref{EQ:minmax_p}) are chosen so that Assumption~\ref{AS:effectiveness} holds. It is easily checked that Assumption \ref{AS:existence_sub} holds. Under these conditions, all the results established in Section \ref{SE:problem} and Section \ref{SE:complexity} hold for problem (\ref{EQ:minmax_p}).

We consider two choices for the vector $\bm L$, 
$${\bm L}_{_{\rm\small GN}}=\big(\tilde L,\ldots, \tilde L\big)\text{ and }{\bm{L}}_{_{\rm MProx}}=(L_1,\ldots,L_m)$$
where $\tilde{{L}}:= \mathrm{max}\{L_1,\ldots,L_m\}$.  {By replacing $\bm{L}$ in \alg\ with ${\bm L}_{_{\rm\small GN}}$}, we recover the PGNM method. Similarly, by 
{replacing $\bm{L}$ in \alg\ with}
${\bm{L}}_{_{\rm MProx}}$, we have the \alg\ algorithm. In both cases, the subproblem $(\mathcal{P}_x)$ writes
\begin{align}
	x_{k+1}\in \underset{y\in\mathbb{R}^n, s \in \RR}{\mathrm{argmin}}\quad&s\ \nonumber\\
	\mathrm{s.t.}\quad&s\geq f_i(x_k) + \langle \nabla f_i(x_k),y - x_k\rangle + \frac{L'_i}{2}\|y - x_k\|^2,\ i=1,\ldots,m,
\label{EQ:equiv_sub_proposed}
\end{align}
the choice of $L'_i$, $i = 1, \ldots, m$, being the only difference between the two methods. Problem (\ref{EQ:equiv_sub_proposed}) is a quadratically constrained quadratic program (QCQP) which can be solved by appropriate solvers. 

As for APGNM, the accelerated variant of PGNM, it requires to solve iteratively QCQP subproblems which are analogous to (\ref{EQ:equiv_sub_proposed}), their solutions can be computed similarly. However APGNM requires to solve two QCQP subproblems at each step (we refer the interested readers to \cite{Drusvyatskiy:20162} for more details on this algorithm).

\subsection{Performance comparison}

In this section, we provide numerical comparison of the performances of \alg, PGNM, and APGNM on problem (\ref{EQ:minmax_p}), with synthetic, randomly generated data. 
First, let us explain the data generation process. The matrix $Q_m$ is set to zero so that the resulting component, $f_m$ is actually an affine function (and $L_m$ is set to $0$). The rest of problem data is generated as follows. For $i = 1,\ldots,m-1$ we generate positive semidefinite matrices $Q_i = Y_i D_i Y_i$, where $Y_i$ are random Householder orthogonal matrices 
$$Y_i = I_n - 2\frac{\omega_i \omega_i^T}{\|\omega_i\|^2},$$
where $I_n$ is the $n\times n$ identity matrix and the coordinates of $\omega_i\in\mathbb{R}^n$ are chosen as
{independent} realizations of a unit Gaussian. For $i = 1,\ldots,m-1$, $D_i$ is chosen as an $n\times n$ diagonal matrix whose diagonal elements are randomly shuffled from the set
$$\{i\times10^{\frac{j}{n-1}}, j = 1,\ldots,n\}.$$
For $i=1,\ldots,m$, the coordinates of $b_i$ are chosen as independent realizations of a Gaussian $N(0,(1/3)^2)$. Finally, we choose 
$$c_i = 10^{2i/m}.$$
Note that for each $i=1,\ldots,m$, the Lipschitz constant of $f_i$ is twice the maximum eigenvalue of $Q_i$, i.e.,
$$L_i = 2 \times {\mathrm{max}}\{\text{eigenvalue}(Q_i)\}.$$

To iteratively solve the surrogate QCQP subproblems, we used {\it MOSEK} solver \cite{Mosek:2016} interfaced with \textit{YALMIP} toolbox \cite{Lofberg:2004} in Matlab. We fix the number of variables to $n =100$ and choose the origin as the initial point. For each $m\in\{5,10,15,20,25,30\}$, we repeat the random data generation process 20 times and run the three algorithms to solve each corresponding problem. 


As a measure of performance, in order to compare efficiency between different random runs of the data generation process, we use the following normalized suboptimality gap $$\left(\frac{g\circ F(x_{k}) -  g\circ F(x_*)} { g\circ F(x_{0}) - g\circ F(x_*)}\right)_{k\in \mathbb{N}}.$$
Statistics for the three algorithms are presented in Table \ref{TAB:comparison}. It is clear from these figures, that \alg\ is dramatically faster than both PGNM and APGNM. Furthermore, \alg\ seems to suffer less from increasing values of $m$. Finally 
{APGNM}
 is less consistent in terms of performances.
\begin{table}[htp]
	\caption{Comparisons of \alg, PGNM, and APGNM in terms of normalized suboptimality gap in percentage at 10th and 20th iterations. The mean and standard deviations represent the central tendency and dispersion for 20 random runs of the data generation process.}
	\begin{center}%
		\begin{tabular}{c|cc||cc||cc}
			\hline
			\multicolumn{7}{c}{}\\
			\multicolumn{7}{c}{$100\times${\large $\frac{g\circ F(x_{k}) -  g\circ F(x_*)}{ g\circ F(x_{0}) - g\circ F(x_*)}$}}\\
			\multicolumn{7}{c}{}\\
			\cline{1-7} 
			& \multicolumn{2}{c||}{ } & \multicolumn{2}{c||}{} & \multicolumn{2}{c}{}  \\
			& \multicolumn{2}{c||}{\alg\ } & \multicolumn{2}{c||}{PGNM } & \multicolumn{2}{c}{APGNM}  \\
			& \multicolumn{2}{c||}{ } & \multicolumn{2}{c||}{} & \multicolumn{2}{c}{}  \\
			\cline{2-7}
 			& \multicolumn{2}{c||}{$k=10$} &  \multicolumn{2}{c||}{$k=10$}  &  \multicolumn{2}{c}{$k=10$}  \\\hline
			$m$ &mean  & std  & mean & std & mean & std \\\hline
			5 & 0.48 & 1.12$\times$10$^{-1}$	&94.44&5.01$\times$10$^{-1}$&88.15&1.07 								\\
			10 & 0.46& 9.81$\times$10$^{-2}$	&95.20&3.34$\times$10$^{-1}$&89.78&7.14$\times$10$^{-1}$\\
			15  & 0.47 & 8.49$\times$10$^{-2}$&95.62&2.84$\times$10$^{-1}$&90.66&6.05$\times$10$^{-1}$\\
			20 &  0.47& 8.73$\times$10$^{-2}$ &95.72&3.71$\times$10$^{-1}$&90.88&7.92$\times$10$^{-1}$\\
			25  & 0.44&  1.07$\times$10$^{-1}$&95.79&4.04$\times$10$^{-1}$&91.03&8.64$\times$10$^{-1}$\\
			30 & 0.46  & 8.01$\times$10$^{-2}$&95.82&4.33$\times$10$^{-1}$&91.09&9.26$\times$10$^{-1}$\\
\hline
 			&  \multicolumn{2}{c||}{$k=20$}  &  \multicolumn{2}{c||}{$k=20$} &  \multicolumn{2}{c}{$k=20$} \\\hline
			$m$ & mean  & std  & mean & std & mean &std\\\hline
			5 &  2.24$\times$10$^{-2}$ & 5.70$\times$10$^{-3}$&88.88&1.00  								&62.34 & 3.40 \\
			10&  2.24$\times$10$^{-2}$ & 5.86$\times$10$^{-3}$&90.40&6.68$\times$10$^{-1}$&67.53 & 2.27 \\
			15&  2.15$\times$10$^{-2}$ & 5.43$\times$10$^{-3}$&91.24&5.67$\times$10$^{-1}$&70.35 & 1.92 \\
			20&  2.13$\times$10$^{-2}$ & 4.67$\times$10$^{-3}$&91.44&7.42$\times$10$^{-1}$&71.03 & 2.51 \\
			25&  2.08$\times$10$^{-2}$ & 6.45$\times$10$^{-3}$&91.57&8.10$\times$10$^{-1}$&71.49 & 2.74 \\
			30&  2.07$\times$10$^{-2}$ & 5.46$\times$10$^{-3}$&91.64&8.67$\times$10$^{-1}$&71.71 & 2.94 \\
\hline
\end{tabular}
\end{center}
\label{TAB:comparison}
\end{table}

\begin{figure}[!htbp]
\begin{center}
\includegraphics[trim=0.0cm 0.15cm 0.4cm 0.1cm, clip=true, width=1.0\textwidth, angle=0]{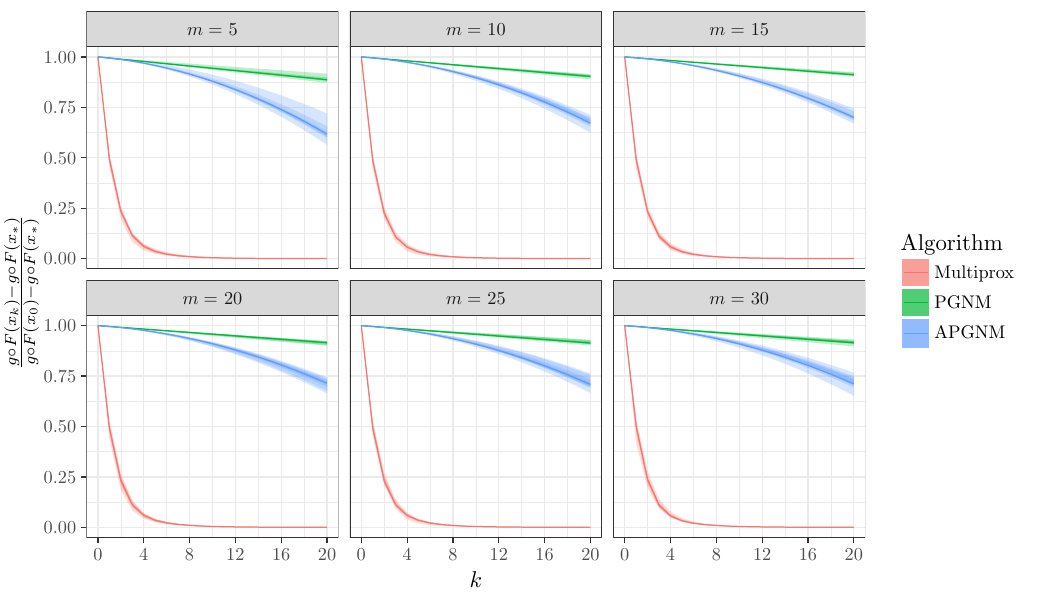}
\caption{Comparisons of \alg, PGNM, and APGNM in terms of normalized suboptimality gap. The central line is the median over the 20 random runs of the data generation process and the ribbons contain 25\%, 50\%, 75\% and 100\% of the simulations respectively.}
\label{Fig:com_gauss_proposed}
\end{center}
\end{figure}
\begin{figure}[!htbp]
\begin{center}
\includegraphics[trim=0.0cm 0.15cm 0.4cm 0.1cm, clip=true, width=1.0\textwidth, angle=0]{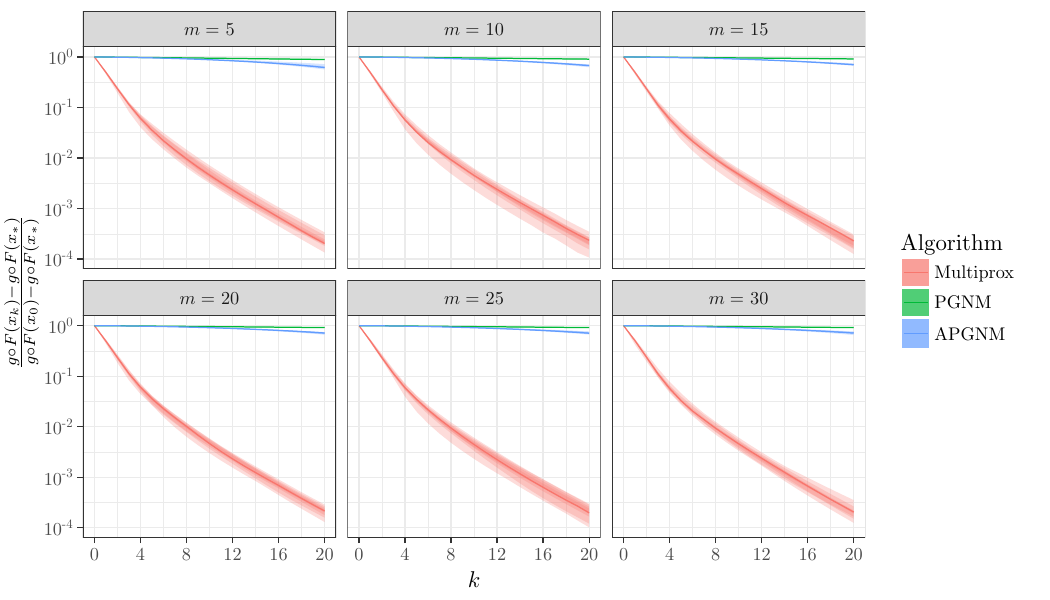}
\caption{Same as Figure \ref{Fig:com_gauss_proposed} with the $y$ axis on log scale.}
\label{Fig:com_gauss_proposed1}
\end{center}
\end{figure}

A graphical view of the same results is presented in Figure \ref{Fig:com_gauss_proposed} and a log-scale view is given in Figure \ref{Fig:com_gauss_proposed1}.
One can see from Figures \ref{Fig:com_gauss_proposed} and  \ref{Fig:com_gauss_proposed1} that the sequences for \alg\ and PGNM are 
nonincreasing as predicted by our theory. Note that the sequence generated by APGNM is not necessarily nonincreasing, although all the 
sequences represented in Figure \ref{Fig:com_gauss_proposed} are strictly decreasing. It is clear that the decreasing slopes for PGNM and 
APGNM are much smaller than that of \alg, coinciding with the  data in Table \ref{TAB:comparison}.  This situation is actually not surprising 
since the data of the problem were chosen in a way ensuring that the $L_i$'s ($i=1,\ldots,m$) can take very different values as in many ill posed problems. The strength of \alg\ is that $\bm{L}$ can be chosen 
appropriately to adapt to this disparity. On the other hand the use of a single parameter (as in PGNM or APGNM) yields smaller steps and thus slower convergence.

{
\section*{Acknowledgements}
We thank Marc Teboulle for his suggestions and the anonymous referees for their very useful comments. And, we thank Radu Ioan Bot for kindly pointing out reference \cite{BGW} after this work has been accepted.
}

\appendix
\section{Proof of Proposition \ref{PR:chainRule}}
\label{SE:appendixProofChainRule}

Let us recall a qualification condition from \cite{Rockafellar:1998}. Given any $x\in F^{-1}(\dom g)$, let
\begin{equation}J(x,\cdot):\left\{
\begin{array}{lll}
 \mathbb{R}^n & \rightarrow & \mathbb{R}^{m}\\
  \omega &\mapsto &  F({x}) + \nabla F({x}) \omega,
\label{EQ:linearized_mapping}
\end{array}\right.
\end{equation}
be the linearized mapping of $F$ at $x$. 
Proposition~\ref{PR:chainRule}  follows immediately from the classical chain rule given in  \cite[Theorem 10.6]{Rockafellar:1998} and the following proposition.

\begin{proposition}[Two equivalent qualification conditions]\label{TH:QC} Under Assumptions \ref{AS:convex} on $F$ and $g$, Assumption~\ref{AS:Qualification} holds if and only if 
\begin{center}
{\bf (QC) }  $\dom g$ cannot be separated from $J(x,\mathbb{R}^n)$ for any $x\in F^{-1}(\dom g)$.
\end{center}
\end{proposition}

\begin{proof} We first suppose that {\bf (QC)} is true. We begin with a remark showing that this implies that $\dom g$ is not empty. Let $A$ and $B$ be two subsets of $\RR^m$. The logical negation of the sentence ``$A$ and $B$ can be separated'' can be written as follows: for all $a$ in $\RR^m$ and for all $b \in \RR$, there exists $y \in A$ such that
\begin{align*}
	{a^Ty} + b > 0,
\end{align*}
or, there exists $z \in B$ such that
\begin{align*}
	{a^T z}
	+ b< 0.
\end{align*}
In particular if $A$ and $B$ cannot be separated, then either $A$ or $B$ is not empty. Note that if $\dom g$ is empty, then so is the set $\{J(x, \RR^n),\, x \in F^{-1}(\dom g)\}$. Hence {\bf (QC)} actually implies that $\dom g$ is not empty.				
				Pick a point $\tilde{x}\in F^{-1}(\dom g)$. If $F(\tilde{x})\in \text{int dom}(g)$, there is nothing to prove, so we may suppose that $F(\tilde{x})\in \text{bd dom}\,g$.
				{If we had $[\text{int dom}(g)] \cap J(\tilde{x},\mathbb{R}^n) = \emptyset$, then, $\dom g$ and $J(\tilde{x},\mathbb{R}^n)$ could be separated by Hahn-Banach theorem contradicting {\bf (QC)}. }
				Hence, there exists $\tilde{\omega}\in\mathbb{R}^n$ such that $J(\tilde{x},\tilde{\omega})\in \text{int dom}(g)$. {Note that, since $F(\tilde{x})\in \mathrm{dom}(g)$ and $g$ is nondecreasing with respect to each argument, it follows $F(\tilde{x}) - d\in \mathrm{dom}(g)$ for any $d\in(\mathbb{R}_+^*)^m$, indicating that $\mathrm{int}(\mathrm{dom}(g))\neq \emptyset$}. Since $\dom g$ is convex, a classical result yields
\begin{eqnarray}
J(\tilde{x},\lambda \tilde{\omega}) \in \text{int dom}(g),\ \forall\,\lambda\in(0,1].
\label{EQ:QC:int}
\end{eqnarray}
On the other hand $F$ is differentiable thus
\begin{eqnarray}
\|F(\tilde{x} + \lambda \tilde{\omega}) - J(\tilde{x},\lambda \tilde{\omega})\| =  o(\lambda),
\label{EQ:Taylor}
\end{eqnarray}
where $o(\lambda)/\lambda$ tends to zero as $\lambda$ goes to zero.

After these basic observations, let us  recall an important property of the signed distance (see \cite[p. 154]{Hiriart-Urruty:1993}).
Let $D\subset \mathbb{R}^m$ be a nonempty closed convex set. Then, the function
\begin{eqnarray}\label{LE:Concavity} D\rightarrow \mathbb{R}_+,\ z\mapsto \mathrm{dist}(z,\mathrm{bd}(D)),\nonumber
\end{eqnarray}
is concave.
Using this concavity property  for $D=\dom g$ and the fact that $F(\tilde{x}) = J(\tilde{x},0)$, it holds that 
\begin{eqnarray}
&&\lambda \text{dist}[J(\tilde{x},\tilde{\omega}),\text{bd}(\dom g)] + (1-\lambda) \text{dist}[F(\tilde{x}),\text{bd}(\dom g)]\nonumber\\
&& \leq \text{dist}[J(\tilde{x},\lambda\tilde{\omega}),\text{bd}(\dom g)],\ \forall\ \lambda\in[0,1].\nonumber
\end{eqnarray}
Since $\text{dist}[F(\tilde{x}),\text{bd}(\dom g)] = 0$, it follows that
\begin{eqnarray}
\lambda \text{dist}[J(\tilde{x},\tilde{\omega}),\text{bd}(\dom g)] \leq \text{dist}[J(\tilde{x},\lambda\tilde{\omega}),\text{bd}(\dom g)],\ \lambda\in[0,1].
\label{EQ:Taylor1}
\end{eqnarray}
Note that $ \text{dist}[J(\tilde{x},\tilde{\omega}),\text{bd}(\dom g)]  > 0$ since $J(\tilde{x},\tilde{\omega})\in \text{int dom}(g)$. Hence, equation (\ref{EQ:Taylor}) indicates that there exists $\epsilon > 0$ such that for any $0 < \lambda \leq \epsilon$, we have
$$\| F(\tilde{x} + \lambda \tilde{\omega})  -  J(\tilde{x},\lambda\tilde{\omega})  \| < \lambda  \text{dist}[J(\tilde{x},\tilde{\omega}),\text{bd}(\dom g)] .$$
Substituting this inequality into equation (\ref{EQ:Taylor1}) indicates that for any $0 < \lambda \leq \epsilon$, we have
$$\|F(\tilde{x} + \lambda \tilde{\omega}) - J(\tilde{x},\lambda\tilde{\omega}) \| < \text{dist}[ J(\tilde{x},\lambda\tilde{\omega}),\text{bd}(\dom g)].$$
Using equation (\ref{EQ:QC:int}), for any  $0 <\lambda \leq \epsilon$, we have $F(\tilde{x} + \lambda \tilde{\omega})\in \text{int dom}(g)$. This shows the first implication of the equivalence. 

\bigskip

{Let us prove the reverse implication by contraposition and }assume that {\bf (QC)} does not hold, that is, there exists a point $\tilde{x}\in F^{-1}(\dom g)$  such that $\dom g$ can be separated from $J(\tilde{x},\mathbb{R}^n)$. In this case, there exists $a \neq 0 \in \mathbb{R}^{m}$ and $b\in\mathbb{R}$ such that
\begin{eqnarray}
\begin{cases}
a^T z + b \leq 0,\ \forall z\in \dom g,\\
a^T J(\tilde{x},\omega) + b \geq 0,\ \forall \omega \in \mathbb{R}^n.
\end{cases}
\label{EQ:separating}
\end{eqnarray}
Since $J(\tilde{x},0) = F(\tilde{x}) \in \dom g$, it follows 
\begin{eqnarray}
a^TJ(\tilde{x},0) + b = 0.
\label{EQ:boundary_point}
\end{eqnarray}
 By the coordinatewise convexity of $F$,  for every $i\in\{1,\ldots,m\}$ one has
$$
f_i(y) \begin{cases}
\geq f_i(x) +
{\nabla f_i(x)^T (y-x)},\ L_i >0,\\
= f_i(x) + 
{\nabla f_i(x)^T(y-x)},\ L_i =0,
\end{cases}\ \ \forall (x,y)\in\mathbb{R}^n\times\mathbb{R}^n.
$$
We thus have the componentwise inequality
$$J(\tilde{x},0) - [F(\tilde{x}+\omega) - J(\tilde{x},\omega) ] \leq F(\tilde x).$$
The monotonicity properties of $g$ implies thus that 
$$J(\tilde{x},0) - [F(\tilde{x}+\omega) - J(\tilde{x},\omega) ] \in \dom g,\ \forall \omega\in\mathbb{R}^n.$$
 As a result, combining equation (\ref{EQ:separating}) with equation (\ref{EQ:boundary_point}), one has
$$a^{T} \big\{ J(\tilde{x},0) - [F(\tilde{x}+\omega) - J(\tilde{x},\omega) ]\} + b \leq  a^T J(\tilde{x},0) + b,\ \forall \omega\in\mathbb{R}^n,$$
which reduces to $a^T [F(\tilde{x}+\omega) - J(\tilde{x},\omega) ] \geq 0,$ $\forall \omega\in\mathbb{R}^n$. Hence, for any $\omega\in\mathbb{R}^n$ one has  
\begin{eqnarray}
a^T F(\tilde{x}+\omega ) + b &=& a^T J(\tilde{x},\omega) + b + a^T(F(\tilde{x}+\omega) - J(\tilde{x},\omega) )\nonumber\\
&\geq& a^T J(\tilde{x},\omega) + b\;\geq\; 0,\nonumber
\end{eqnarray}
where for the last inequality, equation (\ref{EQ:separating}) is used. {{This inequality combined with the fact $a^T z + b < 0,\, \forall z\in \mathrm{int\ dom}\ g$ obtained according to the first item of equation (\ref{EQ:separating}),} shows that $F(\tilde{x}+\omega)\not\in \text{int\,dom}\ g$, for all $\omega\in\mathbb{R}^n$, and thus $F^{-1}(\text{int\,dom}\,g)= \emptyset$, that is Assumption \ref{AS:Qualification} does not hold.}
 This provides the reverse implication and the proof is complete. 
\end{proof}

\section{Proof of Lemma \ref{TH:explicitBound}}
\label{SE:appendixExplicitBound}
In this section, we present an explicit estimate of the condition number appearing in our complexity result. Let us first introduce a notation. For any $D\subset\mathbb{R}^m$, nonempty closed set, we define a signed distance function as
\begin{eqnarray}
D\rightarrow \mathbb{R},\ z\mapsto 
\mathrm{sdist}=
\begin{cases}
\mathrm{dist}(z,\mathrm{bd}(D)),\ \text{if}\ z\in \mathrm{int}(D),\\
-\mathrm{dist}(z,\mathrm{bd}(D)),\ \text{otherwise}.
\end{cases}
\label{EQ:signed_distance}
\end{eqnarray}
It is worth recalling that the signed distance function is concave (see \cite[p. 154]{Hiriart-Urruty:1993}).
We begin with a lemma which describes a monotonicity property of the signed distance function.
\begin{lemma}\label{LE:g_bd}
Given any $z\in \dom g $ and any $d=(d_1,\ldots,d_m)\in \mathbb{R}_+^m$ with $d_i =0$ if $L_i = 0$, {if $\mathrm{bd\ dom}g \neq \emptyset$,} one has 
\begin{eqnarray}
\mathrm{sdist}(z,\mathrm{bd}\ \dom g ) \geq \mathrm{sdist}(z+ d, \mathrm{bd}\ \dom g ).
\label{EQ:ineq_sdist}
\end{eqnarray}
\end{lemma}
\begin{proof}
Fix an arbitrary $z\in \dom g $ and an 
{arbitrary} $d=(d_1,\ldots,d_m)\in\mathbb{R}_+^m$ such that $d_i = 0$ whenever $L_i = 0$ in the sequel of the proof. If $z+d \not\in \mathrm{dom}\ g$, equation (\ref{EQ:ineq_sdist}) holds true by the definition in equation (\ref{EQ:signed_distance}). 

From now on, we suppose $z+d \in \mathrm{dom}\ g$.
Let $\bar{z} \in \mathrm{bd}\ \dom g $ be a point such that
$$\bar{z} \in \underset{\hat{z}\in \mathrm{bd}\ \dom g }{\mathrm{argmin}}\ \ \|z - \hat{z}\|.$$
Then, one has 
\begin{eqnarray}
\|z - \bar{z}\| = \mathrm{sdist}(z,\mathrm{bd}\ \dom g ).
\label{EQ:z_barz}
\end{eqnarray}
Since $\bar{z}$ lies on the boundary of $\dom g $, it follows that $\bar{z}+d \not\in \mathrm{int}\ \dom g $ because of the monotonicity property of $g$ in Assumption \ref{AS:convex}.
{Hence, by the definition of $\mathrm{sdist}$ in equation (\ref{EQ:signed_distance}), we have}
$$\mathrm{sdist}(z+d,\mathrm{bd}\ \dom g ) =  \mathrm{dist}(z+d,\RR^m \setminus\mathrm{int}\ \dom g )  \leq \|(z + d) - (\bar{z} + d)\| = \|z - \bar{z}\|.$$
Combining this inequality with equation (\ref{EQ:z_barz}) completes the proof.
\end{proof}
The following lemma shows that it is possible to construct a convex combination between the current $x$ and the Slater point $\bar{x}$ given in Assumption \ref{AS:Qualification} which will be a Slater point for the current sub-problem with a uniform control over the ``degree'' of qualification.
\begin{lemma}
\label{LE:lower_bound_Hxw}
Let $\bar{x}$ be given as in Assumption \ref{AS:Qualification} and $x \in F^{-1}(\mathrm{dom}\ g)$ {and assume that $\mathrm{bd\ dom}g \neq \emptyset$}. Set 
$$\gamma(\bar{x},x):= \mathrm{min}\left\{1,\frac{\mathrm{sdist}[F(\bar{x}),\mathrm{bd}\ \dom g ]}{2\{\mathrm{sdist}[F(\bar{x}),\mathrm{bd}\ \dom g ]- \mathrm{sdist}[F(\bar{x})+\bm{L} \|x - \bar{x}\|^2/2,\mathrm{bd}\ \dom g ]\}}\right\}.$$
Then, 
\begin{align}
 \mathrm{sdist}[H(x,x+\gamma(\bar{x},x)(\bar{x} - x)),\mathrm{bd}\ \dom g ]&\geq\frac{\mathrm{sdist}[F(\bar{x}),\mathrm{bd}\ \dom g ]^2 }{4 \big\{\mathrm{sdist}[F(\bar{x}),\mathrm{bd}\ \dom g ] + \|\bm{L}\| \|x - \bar{x}\|^2/2\big\}}.
\label{EQ:lemma_low_bound_h}
\end{align}
\end{lemma}
\begin{proof}
Fix an arbitary $x\in  K$. Then, for any $t\in (0,1]$ one has
\begin{eqnarray}
H(x,x+t(\bar{x} - x)) &=& F(x) + t\nabla F(x)(\bar{x} - x) + \frac{\bm{L}}{2}\|\bar{x} - x\|^2 t^2\nonumber\\
&\leq& (1-t) F(x) + t F(\bar{x}) + \frac{\bm{L}}{2}\|x - \bar{x}\|^2 t^2,
\label{EQ:ineq1_t}
\end{eqnarray}
where the last inequality is obtained by applying the coordinatewise convexity of $F$.
Therefore, for any $t\in(0,1]$, we have 
\begin{eqnarray}
&& \mathrm{sdist}[H(x,x+t(\bar{x} - x)),\mathrm{bd}\ \dom g ]\nonumber\\
 &\overset{(a)}{\geq}& \mathrm{sdist}[(1-t) F(x) + t F(\bar{x}) + \frac{\bm{L}}{2} \|x - \bar{x}\|^2 t^2,\mathrm{bd}\ \dom g ]\nonumber\\
&\overset{(b)}{\geq}&  \mathrm{sdist} [F(x),\mathrm{bd}\ \dom g ](1-t) +  \mathrm{sdist} [F(\bar{x}) + t \frac{\bm{L}}{2}\|x - \bar{x}\|^2 ,\mathrm{bd}\ \dom g ] t \nonumber\\
&\overset{(c)}{\geq} &   \mathrm{sdist} [(1-t) F(\bar{x}) + t ( F(\bar{x}) + \frac{\bm{L}}{2} \|x - \bar{x}\|^2 ),\mathrm{bd}\ \dom g ] t \nonumber\\
&\overset{(d)}{\geq}&  \mathrm{sdist}[F(\bar{x}),\mathrm{bd}\ \dom g ] t(1-t) +  \mathrm{sdist}[ F(\bar{x}) + \frac{\bm{L}}{2} \|x - \bar{x}\|^2,\mathrm{bd}\ \dom g ] t^2 \nonumber\\
&=:& \delta(t),
\label{EQ:ineq_delta}
\end{eqnarray}
where for (a) we combine equation (\ref{EQ:ineq1_t}) with Lemma \ref{LE:g_bd}, for (b) we use the concavity of the signed distance function (see \cite[p. 154]{Hiriart-Urruty:1993}), for (c) we use the fact that $(1-t) \mathrm{sdist} [F(x),\mathrm{bd}\ \dom g ] \geq 0$, and for (d) we use the concavity of the signed distance function again.

It is easy to verify that $\gamma(\bar{x},x)\in(0,1]$ is the maximizer of $\delta(t)$ over the interval $(0,1]$. We now consider the following inequality. 
\begin{eqnarray}
				\label{EQ:ineqDistance}
\mathrm{sdist}[F(\bar{x})+\bm{L} \|x - \bar{x}\|^2/2,\mathrm{bd}\ \dom g ] \geq - \frac{\|\bm{L}\|}{2}\|x - \bar{x}\|^2, 
\end{eqnarray}
Inequality (\ref{EQ:ineqDistance}) holds true: indeed, either $F(\bar{x})+\bm{L} \|x - \bar{x}\|^2/2 \in \mathrm{dom}\ g$ and the result is trivial or otherwise, the result holds by the definition of the distance as an infimum.
If $\gamma(\bar{x},x) = 1$, by its definition, one immediately has
\begin{eqnarray}
\frac{\mathrm{sdist}[F(\bar{x}),\mathrm{bd}\ \dom g ]}{2\{\mathrm{sdist}[F(\bar{x}),\mathrm{bd}\ \dom g ]- \mathrm{sdist}[F(\bar{x})+\bm{L} \|x - \bar{x}\|^2/2,\mathrm{bd}\ \dom g ]\}} \geq 1,\nonumber
\end{eqnarray}
which implies
\begin{eqnarray}
\mathrm{sdist}[F(\bar{x})+\bm{L} \|x - \bar{x}\|^2/2,\mathrm{bd}\ \dom g ] \geq \frac{\mathrm{sdist}[F(\bar{x}),\mathrm{bd}\ \dom g ]}{2}.
\label{EQ:beta_alpha2}
\end{eqnarray}
Substituting $\gamma(\bar{x},x) = 1$ into equation (\ref{EQ:ineq_delta}) yields
\begin{eqnarray}
\delta(\gamma(\bar{x},x)) &=&   \mathrm{sdist}[ F(\bar{x}) + \frac{\bm{L}}{2} \|x - \bar{x}\|^2,\mathrm{bd}\ \dom g ] \nonumber\\
&\geq & \frac{\mathrm{sdist}[F(\bar{x}),\mathrm{bd}\ \dom g ]}{2}\nonumber\\
&\geq& \frac{\mathrm{sdist}[F(\bar{x}),\mathrm{bd}\ \dom g ]}{2} \times \frac{1}{2}\frac{\mathrm{sdist}[F(\bar{x}),\mathrm{bd}\ \dom g ]}{\mathrm{sdist}[F(\bar{x}),\mathrm{bd}\ \dom g ] + \|\bm{L}\|\|x - \bar{x}\|^2/2}\nonumber\\
&=&  \frac{\mathrm{sdist}[F(\bar{x}),\mathrm{bd}\ \dom g ]^2 }{4 \big\{\mathrm{sdist}[F(\bar{x}),\mathrm{bd}\ \dom g ] + \|\bm{L}\| \|x - \bar{x}\|^2/2\big\}},
\label{EQ:gamma=1}
\end{eqnarray}
where the {first}
 inequality is obtained by considering equation (\ref{EQ:beta_alpha2}). As a result, equation (\ref{EQ:lemma_low_bound_h}) holds true if $\gamma(\bar{x},x) = 1$.

From now on, let us consider $\gamma(\bar{x},x)<1$. In this case, one has
 \begin{eqnarray}
\gamma(\bar{x},x) = \frac{\mathrm{sdist}[F(\bar{x}),\mathrm{bd}\ \dom g ]}{2\{\mathrm{sdist}[F(\bar{x}),\mathrm{bd}\ \dom g ]- \mathrm{sdist}[F(\bar{x})+\bm{L} \|x - \bar{x}\|^2/2,\mathrm{bd}\ \dom g ]\}}.\nonumber
\end{eqnarray}
Substituting into equation (\ref{EQ:ineq_delta}) and using (\ref{EQ:ineqDistance}) leads to 
\begin{eqnarray}
\delta(\gamma(\bar{x},x)) &=& \frac{\mathrm{sdist}[F(\bar{x}),\mathrm{bd}\ \dom g ]^2 }{4 \big\{\mathrm{sdist}[F(\bar{x}),\mathrm{bd}\ \dom g ] -\mathrm{sdist}[F(\bar{x})+\bm{L} \|x - \bar{x}\|^2/2,\mathrm{bd}\ \dom g ]\big\}}\nonumber\\
&\geq& \frac{\mathrm{sdist}[F(\bar{x}),\mathrm{bd}\ \dom g ]^2 }{4 \big\{\mathrm{sdist}[F(\bar{x}),\mathrm{bd}\ \dom g ] + \|\bm{L}\| \|x - \bar{x}\|^2/2\big\}}.\nonumber
\end{eqnarray}
Eventually, combining this equation with (\ref{EQ:ineq_delta}) and (\ref{EQ:gamma=1}) completes the proof.
\end{proof}
We are now ready to describe the proof of Lemma \ref{TH:explicitBound}
\begin{proof}[Proof of Lemma \ref{TH:explicitBound}]
				(i) As the function $g$ is $L_g$ Lipschitz continuous on its domain, an immediate application of the Cauchy-Schwartz inequality leads to $\bm{L}^T\nu \leq L_g \|\bm{L}\|$ (see also Section \ref{SE:LIP}).

(ii)
{The claim is trivial if $\mathrm{bd\ dom}(g) = \emptyset$, hence we will assume that it is not so that we can use Lemmas 5 and 6.}
Set $w = x + \gamma(\bar{x},x)(\bar{x} - x)$ with $\gamma(\bar{x},x)$ given as in Lemma \ref{LE:lower_bound_Hxw}. By Lemma \ref{LE:lower_bound_Hxw}, one has $w \in \mathrm{dom}(g\circ H(x,\cdot))$. Then, one obtains
\begin{eqnarray}
\frac{\bm{L}^T\nu}{2}\|w - y\|^2 &=& [H(x,w) - H(x,y)]^T\nu\nonumber\\
&\leq& g\circ H(x,w) - g\circ H(x,y)\nonumber\\
&\leq& L_g \|H(x,w) - H(x,y)\|,
\label{EQ:Lipschitz_g_Lg}
\end{eqnarray}
where the equality follows from equation (\ref{EQ:strong_a}), the first inequality is obtained by the convexity of $g$, and the last inequality is due to the assumption that $g$ is $L_g$ Lipschitz continuous on its domain.
On the other hand, a direct calculation yields
\begin{align}
\|H(x,w) - H(x,y)\| =&\ \|\nabla F(x) (w - y) + \frac{\bm{L}}{2}\|w - x\|^2 - \frac{\bm{L}}{2}\|y - x\|^2 \|\nonumber\\ 
=&\ \|\nabla F(x) (w - y) + \frac{\bm{L}}{2} ( \|w - y\|^2  + 2(w-y)^T(y-x) ) \|\nonumber\\ 
\leq&\ \|\nabla F(x)\|_{\rm op}\|w-y\| + \frac{\|\bm{L}\|}{2} \|w - y\|^2 + \|\bm{L}\|\|w-y\|\|y - x\|\nonumber\\
=&\ \|w - y\| \left[ \|\nabla F(x)\|_{\rm op} + \frac{\|\bm{L}\|}{2}\|w - y\| + \|\bm{L}\|\| y - x\|\right]\nonumber\\
\leq&\  \|w - y\| \left[ \|\nabla F(x)\|_{\rm op} + \frac{\|\bm{L}\|}{2}\|x - y\| + \gamma(\bar{x},x)\frac{\|\bm{L}\|}{2}\|\bar{x} - x\| + \|\bm{L}\|\| y - x\|\right]\nonumber\\
\leq&\  \|w - y\| \left[ \|\nabla F(x)\|_{\rm op} + \frac{3\|\bm{L}\|}{2}\|x - y\| + \frac{\|\bm{L}\|}{2}\|\bar{x} - x\| \right]\nonumber
\end{align}
Substituting this inequality into equation (\ref{EQ:Lipschitz_g_Lg}) leads to
\begin{eqnarray}
				\frac{\bm{L}^T \nu}{2}\|H(x,w) - H(x,y)\| \leq L_g \left[\|\nabla F(x)\|_{\rm op} + \frac{3\|\bm{L}\|}{2}\|x - y\| + \frac{\|\bm{L}\|}{2}\|\bar{x} - x\| \right]^2.
\end{eqnarray}
As $H(x,y)$ is on the boundary of $\dom g $, it follows that
\begin{eqnarray}
\|H(x,w) - H(x,y)\| \geq \mathrm{sdist}[H(x,w),\mathrm{bd}\ \dom g ].
\end{eqnarray}
Combining this inequality with Lemma \ref{LE:lower_bound_Hxw} eventually completes the proof. 
\end{proof}

\end{document}